\documentclass[11pt, reqno]{amsart}
\usepackage{amsaddr}
\usepackage[english]{babel}
\usepackage{color}
\usepackage{amsmath, latexsym, amsthm, amsfonts,bm,amssymb,mathscinet} 
\usepackage{xifthen}
\usepackage[text={5.8in,8.3in},centering]{geometry}
\usepackage{epsfig}
\usepackage[colorlinks,linkcolor=red, citecolor=red,urlcolor=black]{hyperref}
\usepackage[dvipsnames]{xcolor}
\usepackage{natbib}
\usepackage{enumitem}
\usepackage{verbatim}
\usepackage{footmisc}
\usepackage{cleveref}
\usepackage[ruled, linesnumbered]{algorithm2e}
\usepackage{algpseudocode}

\crefname{algocf}{algorithm}{algorithm}
\Crefname{algocf}{Algorithm}{Algorithms}

\numberwithin{equation}{section}

\newcommand{\dd}{\,\mathrm{d}}
\renewcommand{\scr}[1]{{\mathcal #1}}
\newcommand{\bb}[1]{{\mathbb #1}}
\newcommand{\argmin}{\operatornamewithlimits{argmin}}

\newcommand{\Bm}{\begin{pmatrix}}
	\newcommand{\Em}{\end{pmatrix}}
\newcommand{\T}{{\prime}}
\newcommand{\der}[2]{ \ifthenelse{\isempty{#1}}{\frac{\dd}{\dd #2}}{\frac{\dd #1}{\dd #2}} }
\newcommand{\pder}[2]{ \ifthenelse{\isempty{#1}}{\frac{\partial}{\partial #2}}{\frac{\partial #1}{\partial #2}} }
\newcommand{\ind}{\mathbf{1}}

\newcommand{\dto}{\rightsquigarrow}
\newcommand{\eps}{\varepsilon}
\theoremstyle{definition}
\newtheorem{thm}{Theorem}[section]
\newtheorem{prop}[thm]{Proposition}
\newtheorem{lem}[thm]{Lemma}
\newtheorem{cor}[thm]{Corollary}
\newtheorem{rem}[thm]{Remark}
\newtheorem{ex}[thm]{Example}
\newtheorem{defn}[thm]{Definition}
\newtheorem{ass}[thm]{Assumption}

\newcommand{\abs}[1]{\ensuremath{\left\lvert #1 \right\rvert}}

\newcommand{\EE}{\mathbb{E}}

\title{%
Guided simulation of conditioned chemical reaction networks}
\date{\today}
\author{Marc Corstanje  \qquad Frank van der Meulen }

\address{Department of Mathematics, Vrije Universiteit Amsterdam}

\begin{document}
	\maketitle
	\let\thefootnote\relax\footnotetext{Contact: \href{mailto:M.A.Corstanje@vu.nl}{M.A.Corstanje@vu.nl}}
	\begin{abstract}
		
		Let $X$ be a chemical reaction process, modeled as a multi-dimensional continuous-time jump process.  Assume that at given times $0< t_1 < \cdots <t_n$, linear combinations $v_i = L_i X(t_i),\, i=1,\dots ,n$ are observed for given matrices $L_i$. 
We show how  the process that is conditioned on hitting the states $v_1,\dots, v_n$ is obtained by  a change of measure on the law of the unconditioned process. This results in an algorithm for obtaining weighted samples from the conditioned process. Our results are illustrated by  numerical simulations.
	
	 \
		
		
	\
	
	\end{abstract}

	\section{Introduction}
	\label{sec:Introduction}

	Chemical reaction networks are used to study a wide class of biological, physical and chemical processes that evolve over time. For instance, one can think of the transcription of genes to mRNA and then the translation to protein, the kinetics of a virus or the dynamics of chemical components reacting with each other. The forward evolution of such processes can be described in different ways: {\it (i)} a system of ordinary differential equations, see e.g.\ \cite{feinberg2019}, {\it (ii)} a system of stochastic differential equations, see e.g.\ \cite{fuchs2013} or {\it (iii)} continuous-time Markov jump processes, as in \cite{anderson2007}.   It is the third option that we consider in this paper. 
	
	\subsection{Chemical reaction networks}
	\label{subsec:Basic-model}
	Chemical reactions are described as linear combinations of chemical components merging into each other. Typically, one denotes a reaction in which components $A$ and $B$ are merged into $C$ and $D$ by 
	\begin{equation}\label{eq:exABCD} A+B \to C + D.\end{equation}
	More generally, a \textit{chemical reaction network} consists of:
	\begin{itemize}
		\item a \textit{Species} set $\scr{S} = \left\{S_1, \dots, S_d\right\}$ which consists of the chemical components whose counts we model; 
		\item a \textit{Reaction} set $\scr{R}$. A reaction $\ell\in\scr{R}$ is modeled as $\sum_{k}\nu_{k\ell}S_k\to \sum_k \nu_{k\ell}' S_k$. We characterize $\ell$ by the change of counts of the species $\xi_\ell = \left( \nu_{k\ell}'-\nu_{k\ell}\right)_{k=1}^d$. 
	\end{itemize}
	For example, in \eqref{eq:exABCD}, the species set is $\{A,B,C,D\}$ and $\xi$ is given by $(-1,-1,1,1)$. 
	Let $\bb{S}\subseteq\bb{Z}_{\geq 0}^d$ denote the state space of the reaction system. A vector in $\bb{S}$ is of the form $(x_k)_{k=1}^d$, where $x_k$ denotes the species count of type $S_k$. We study a Markov process $X = \left(X(t)\right)_{t\geq 0}$ on $\bb{S}$ that models the evolution of species counts over time. That is, at time $t$, the $k$-th element of  $X(t)$ is given by $\# S_k(t)$. We assume that the initial state $X(0)=x_0$ of the process is known. At a given time $t$, the time until a reaction of type $\ell$ takes place is given by $\tau_\ell$.  When reaction $\hat{\ell} = \argmin_{\ell\in\scr{R}}\tau_\ell$ occurs, the process jumps at time $t+\tau_{\hat\ell}$ to  $X\left(t+\tau_{\hat{\ell}}\right) = X(t)+\xi_{\hat\ell}$. The reactions are assumed to occur according to an inhomogeneous Poisson process with intensity function that we  refer to as the \textit{reaction rate}. A more detailed description of the stochastic model is given in \Cref{sec:GeneralSetting}. The process $X$ evolving on the chemical reaction network is referred to as the \textit{chemical reaction process}. 

	\subsection{Statistical problem}
	\label{subsec:intro-StatisticalProblem}

Suppose at fixed times $0<t_1 <\cdots <t_n$, we  observe $v_1,\dots v_n$, where $v_k=L_k X(t_k)$ with $L_k\in\bb{R}^{m_k\times d}$ and $m_k \leq d$, $k=1,\dots n$. Not assuming $L$ to be the identity matrix is for example important in applications where the measuring device cannot distinguish two or more species, so that only sums of  their counts are observed.  We will assume the rows of each $L_k$ to be linearly independent.  Typically reaction rates are unknown and we wish to infer those from the data. Suppose the reaction rate depend on a parameter vector $\theta$. Likelihood-based inference for $\theta$ is hampered by the lack of closed-form expressions for the  transition probabilities of $X$. However, if the process were observed continuously over time, the problem would be easier. Therefore, it is natural to employ a data-augmentation  scheme where we iteratively sample $X$ on $[0,t_n]$ conditional on $v_1,\dots, v_n$ and $\theta$ and then update $\theta$ conditional on $X$. In this paper, we  focus on the first step, sampling from $\left(X \mid L_k X(t_k)=v_k,\,  k=1,\dots n\right)$. It is a key objective of this paper to show rigorously how this can be done efficiently. 	Note that a simple rejection sampling scheme where we discard all paths contradicting the observations is valid but very inefficient in most settings.

	\subsection{Approach: conditioning by guiding}
		Our approach builds on earlier work in \cite{corstanje2021conditioning} for general Markov processes in case of a single observation. Let us highlight the main points. The law of the process $X$, conditioned to be in a given state at fixed times, is obtained through  Doob's $h$-transform. That is, there is a function $h\colon[0,T]\times\bb{S}\to\mathbb{R}_+$ that depends on the transition probabilities of $X$ which induces a change of measure. Under the new measure, $\bb{P}^{h}$, the process is conditioned to hit the observed states at  times of observation. Since $h$ is typically unknown, we replace it by a fully tractable function $g\colon[0,T]\times\bb{S}\to\mathbb{R}_+$ that itself induces a change of measure to a measure $\bb{P}^{g}$. Under certain conditions, $\bb{P}^{h}$ is absolutely continuous with  respect to $\bb{P}^{g}$ and
		\begin{equation}\label{eq:keyresult}
 		\der{\bb{P}^{h}}{\bb{P}^{g}}(X)=\frac{1}{h(0,x_0)}F(X). 
		\end{equation}
		Here, $F$ is known in closed form, depends on $g$ but does not depend on $h$. Weighted samples of the conditioned process can therefore be obtained by sampling under $\bb{P}^g$. 
The argument for making the above precise is not too hard in the case where $g$ is bounded and bounded away from zero. However, some natural choices we discuss and have been proposed in the literature require a more delicate argument.

\subsection{Related literature}
	Statistical inference for chemical reaction processes has received considerable attention over the past decade. In this section we summarise related work, while in the next section we highlight contributions of  this paper. 
 \cite{Rathinam2021}  consider the setting where one observes a subset of the species counts \emph{continuously} over time and wants to filter the latent species counts. 
\cite{Reeves2022} parametrise  model transition rates  by  neural networks, while assuming all trajectories are \emph{fully continuously} observed. Parameter estimation is then done by gradient ascent to maximise the log likelihood. 

In this paper, we consider \emph{discrete-time} observations and therefore the works below are more closely related to our work. \cite{warne2019} give an introduction to chemical reaction processes and consider estimation for discrete-time partial observations with Gaussian noise focussing on Approximate Bayesian Computation.  
\cite{fearnhead2008computational} and \cite{golightly2019efficient} are probably closest related to our approach. The common starting points of these works is a slightly informal computation that reveals how the reaction rate of the chemical reaction processes changes upon conditioning the process on a future observation (we present this argument at the start of Section \ref{sec:doobh}). While the reaction rate for the conditioned process is intractable, it can be approximated and this simply boils down to choosing $g$ as above. \cite{fearnhead2008computational} approximates $g$ using Euler discretisation of the Chemical Langevin Equation (CLE) assuming the process is fully observed without error.  \cite{golightly2019efficient}
		consider conditioning on a partial observation corrupted by Gaussian noise. Their choice of $g$ is based on the linear noise approximation to the CLE. This is shown to outperform 
		the approach of \cite{fearnhead2008computational} and earlier work in \cite{golightly2015bayesian}. 
		
		\cite{georgoulas2017unbiased} construct an unbiased estimator for the likelihood using  random truncations and computation of matrix exponentials. This in turn is used to exploit the  pseudo-marginal MCMC algorithm (\cite{andrieu2009pseudo}) for parameter estimation. 
Building upon this work \cite{sherlock2023exact}  introduce the 
minimal extended statespace algorithm  and the nearly minimal
extended statespace algorithm  to alleviate the problem of choosing a proposal distribution for the truncation level, as required in \cite{georgoulas2017unbiased}. 
	
	\cite{alt2023entropic} consider the same setting as we do and derive approximations to the filtering and smoothing distributions using expectation propagation.

	\subsection{Contribution}
	
	We provide sufficient conditions on $g$ such that \eqref{eq:keyresult} holds true. 
	We extend the result in \cite{corstanje2021conditioning} for a single complete observation to multiple partial observations in the context of chemical reaction processes. Moreover, we discuss a variation of the \emph{next reaction} algorithm by \cite{gillespie1976} for sampling from a class of reaction networks with unbounded time-dependent reaction intensities.
	
The proposed methods fit within the framework of  
	Backward Filtering Forward Guiding (\cite{vandermeulen2020automatic}), drawing strongly on techniques for  exponential changes of measure as outlined in  \cite{palmowski2002}. This enables us to construct a \emph{guided} process that at any time  takes into account \emph{all} future conditionings.

Compared to	 \cite{fearnhead2008computational}
	 and \cite{golightly2019efficient}, we consider the setting of multiple future conditionings (rather than one), without imposing extrinsic noise on the observations. Moreover, we derive the conditioned process and likelihood ratio in \eqref{eq:keyresult} on path space. Sufficient conditions  on $g$ to guarantee absolute continuity are given in Theorem \ref{thm:absolute-continuity-general}. It turns out that the choices for $g$ in \cite{fearnhead2008computational} and  \cite{golightly2019efficient} satisfy the assumptions of this theorem. 
In numerical examples we show that flexibility in choosing $g$  is particularly beneficial {in cases where} some of the components of the chemical reaction process have counts that vary monotonically over time.

\subsection{Outline} 
	\label{subsec:Introduction-Outline}
	We introduce stochastic chemical reaction processes in \Cref{sec:GeneralSetting} and discuss examples that we will study. In \Cref{sec:doobh}, we describe the approach of conditioning chemical reaction processes by guiding  and present conditions on $g$ such that {$\bb{P}^h$ is absolutely continuous with respect to $\bb{P}^g$}. In \Cref{sec:choices_for_g} we consider various choices for $g$. Conditions for equivalence of $\bb{P}^h$ and $\bb{P}^g$ are discussed in Section \ref{sec:equivalence}.  In Section \ref{sec:simulations}   we present  methods for simulation of conditioned chemical reaction processes together with numerical illustrations.  We end with a discussion section. The appendix contains various proofs.

	\subsection{Frequently used notation}
	\label{subsec:Introduction-Notation}
	Throughout, we assume that we have an underlying probability space $\left(\Omega, \scr{F}, \bb{P}\right)$.
For a stochastic process $X$, we use the notation $X^t = \{ X(s) \colon s\leq t\}$. Given $L \in \bb{R}^{m\times d}$ with $m\leq d$ and $v\in \bb{R}^m$, we denote the inverse image of $v$ under $L$ by $L^{-1}v = \{x\in\bb{R}^d : Lx=v\}$. For functions $f_1(t,x)$ and $f_2(t,x)$ of time and space, we say that $f_1\propto f_2$ if there exists a differentiable function $\kappa$ of time such that for all $t,x$, $f_1(t,x) = \kappa(t)f_2(t,x)$. 
	Derivatives with respect to a variable representing $t$, say $\partial / \partial t$ are denoted by $\partial_t$. 
We denote by 
\begin{equation}
\label{eq:defAn}
A_n=\{ L_kX(t_k)=v_k,\, k=1,\dots,n \},
\end{equation}
a set of conditionings. If a measure $\mu$ is absolutely continuous with respect to $\nu$, we write $\mu \ll \nu$. 

\section{Chemical reaction processes}
	\label{sec:GeneralSetting}
	We construct a stochastic process to model the dynamics of the chemical reaction network described in \Cref{subsec:Basic-model} following Chapter 1 of \cite{anderson2015stochastic}. Let $X$ be a Markov process on $\bb{S}$ such that the $i$-th component of $X(t)$, $X_i(t)$, represents the {frequency} of species $S_i$ at time $t$. A reaction $\ell$ is represented through a difference vector $\xi_\ell\in\bb{S}$ and an intensity $\lambda_\ell\colon[0,\infty)\times\bb{S}\to[0,\infty)$. We assume 	\begin{equation}
		\label{eq:JumpProbabilities}
		\bb{P}\left(X({t+\Delta})-X(t) = \xi_\ell\mid \scr{F}_t^X\right) = \lambda_\ell(t, X(t))\Delta + o(\Delta), \qquad \Delta\downarrow 0, 
	\end{equation}
where $\scr{F}_t^X = \sigma(X^t)$. 
	The jump probabilities specified in \eqref{eq:JumpProbabilities} correspond to a  process with jumps $(\xi_\ell)_{\ell\in\scr{R}}$ and  jump rate functions $(\lambda_\ell)_{\ell\in\scr{R}}$. 
 Throughout, we impose the following assumptions on the network. 
	\begin{ass}
		\label{ass:AssumptionsOnProcess}
			$(\lambda_\ell)_{\ell\in\scr{R}}$ and $(\xi_\ell)_{\ell\in\scr{R}}$ are such that
			\begin{enumerate}[label={\color{red} (\ref{ass:AssumptionsOnProcess}\alph*)}]
				\item \label{ass:AssumptionsOnProcess-NonnegativeRates}
				$\lambda_\ell(t, x)\geq 0$ for all $t\geq 0$, $x\in\bb{S}$ and $\ell\in\scr{R}$. 
				\item \label{ass:AssumptionsOnProcess-WellDefinedJumps}
				For all $\ell\in\scr{R}$, $\xi_\ell\in\bb{Z}^d$ is such $\lambda_\ell(t,x)>0$ implies $x+\xi_\ell\in\bb{S}$ for all $t\geq 0$ and $x\in\bb{S}$.
				\item \label{ass:AssumptionsOnProcess-FiniteSum}
				For all $t\geq 0$ and $x\in\bb{S}$, $\int_0^t\sum_{\ell\in\scr{R}}\lambda_\ell(s,x) \dd s <\infty$.
			\end{enumerate}
	\end{ass}
For a stochastic process $X(t)$, let $T_K=\inf\{ t : |X(t)|\geq  K\}$ and $T_\infty=\lim_{K\to\infty} T_K$. 
Let $X$ be the jump process with jumps $(\xi_\ell)_{\ell\in\scr{R}}$ and  jump rate functions $(\lambda_\ell)_{\ell\in\scr{R}}$ satisfying 
		\begin{equation}
			\label{eq:FormOfX}
			 X(t) = x_0 + \sum_{\ell\in\scr{R}} \xi_\ell Y_\ell\left(\int_0^t \lambda_\ell(s, X(s))\dd s \right), \qquad  0\leq t< T_{\infty} ,
		\end{equation}
	where $(Y_\ell)_{\ell\in\scr{R}}$ are independent, unit rate Poisson processes. We assume the process to be non-explosive: $\mathbb{P}(T_\infty<\infty)=0$. For $f\colon [0,T] \times \bb{S} \to \bb{R}$ with $f(t,\cdot)$ finitely supported for every $t$  define 
\begin{equation}
		\label{eq:Generator} (\scr{L}f)(t,x) =  \sum_{\ell \in\scr{R}} \lambda_\ell(t, x) \left[ f(t,x+\xi_\ell)-f(t,x)\right], 
\end{equation}
By Theorem 1.22 in \cite{anderson2015stochastic}, there exists a filtration $(\scr{F}_t)_{t\geq 0}$ such that for all such $f$ 
 	\[ D^f(t):=f(t, X(t))-\int_0^t (\partial_s + \scr{L}) f(s, X(s))\dd s \]
	is an $(\scr{F}_t)$-martingale. 
	That is, $X$ is the unique  solution to the martingale problem for $\scr{A}:= \partial_t + \scr{L}$.

\subsection{Distribution of  reaction times}
	\label{subsec:Time-distriubutions}
To the $\ell$-th reaction, we associate a reaction time $\tau_\ell$, with distribution specified by
	\begin{equation} 
		\label{eq:JumpDistribution-Inhomogeneous}
		\bb{P}\left(\tau_\ell > \Delta \mid X(t)=x\right) = \exp\left(-\int_t^{t+\Delta} \lambda_\ell (s,x)\dd s\right), \qquad t\geq 0, \Delta>0, x\in\bb{S}.
	\end{equation}

	If $\lambda_\ell$ is constant in time, it follows from \eqref{eq:JumpDistribution-Inhomogeneous} that $\tau_\ell \mid X(s)=x\sim \mathrm{Exp}\left(\lambda_\ell(x)\right)$. This implies that the time $\tau=\min_{\ell\in\scr{R}}\tau_\ell$ that the first reaction occurs satisfies $\tau\mid X(s)=x \sim\mathrm{Exp}\left(\sum_{\ell\in\scr{R}} \lambda_\ell(x)\right)$. \\
	
	If the reaction times are inhomogeneous in time and \Cref{ass:AssumptionsOnProcess} is satisfied, similarly the time for the first reaction to occur has distribution \eqref{eq:JumpDistribution-Inhomogeneous} with intensity function $\sum_{\ell\in\scr{R}}\lambda_\ell$. Moreover, the probability distribution of the first reaction in any subset $R\subseteq\scr{R}$ also satisfies \eqref{eq:JumpDistribution-Inhomogeneous} with intensity function $\sum_{\ell\in R}\lambda_\ell$.
	
	\subsection{Chemical master equation and chemical Langevin equation}
	Denote the transition probabilities of \eqref{eq:FormOfX} by $p$. That is $p(s,x;t,y) = \bb{P}\left(X(t)=y\mid X(s)=x\right)$. The Kolmogorov forward equation yields the \textit{chemical master equation} given by 
	\[ \partial_t p(s,x;t,y) = \sum_{\ell\in\scr{R}}\lambda_\ell(t,y-\xi_\ell)p(s,x;t,y-\xi_\ell) - \sum_{\ell\in\scr{R}} \lambda_\ell(t,y)p(s,x;t,y), \]
	with initial condition $p(s,x;s,y)=1\{y=x\}$. It is well-known, see e.g. \cite{li2020} that the chemical reaction process can be approximated by  solutions to the \textit{Chemical Langevin Equation (CLE)}, which is the SDE 
	\begin{equation}
		\label{eq:CLE}
		\dd Y(t) = b_{\mathrm{CLE}}(t,Y(t)) \dd t + \sigma_{\mathrm{CLE}}(t,Y(t)) \dd W_\ell(t),\qquad Y(0)=x_0
	\end{equation}
	where (assuming the reactions to be numbered $1,\ldots,{B}$) 	\begin{equation}\label{eq:cle_drift_diffusion} \begin{split}
	b_{\mathrm{CLE}}(t,x)&= \sum_{\ell=1}^{B} \lambda_\ell(t,x)\xi_\ell \\   \sigma_{\mathrm{CLE}}(t,x) &= \begin{bmatrix}
		 \xi_1 & \hdots & \xi_{B}
	\end{bmatrix} \sqrt{\mbox{diag}(\lambda_1(t,x),\hdots, \lambda_{{B}}(t,x))}  
	\end{split}
	\end{equation} and 
		 $W$ is an independent $\bb{R}^{B}$-valued Brownian motion.

	\subsection{Examples}
	\begin{ex}[Pure death process]
		\label{ex:deathprocess}
		Our simplest example models a population of initial size $x_0$ in which an individual dies in a time interval $(t,t+\Delta)$ with probability $c\Delta + o(\Delta)$ for some constant $c>0$ and $\Delta$ small. Such a process is modelled as chemical reaction process with just one specie and one reaction, namely {$\lambda_1\colon (t,x)\mapsto cx$ with $\xi_1 = -1$}. 
		\end{ex}

	\label{subsec:introduction-examples}
	\begin{ex}[Gene Transcription and Translation \textit{(GTT)}]
		\label{ex:GTT}
			A stochastic model for the process in which information is encoded in DNA and transferred to mRNA is described in Section 2.1.1 of \cite{anderson2015stochastic}. The basic model considers the three species \textit{Gene (G)}, \textit{mRNA (M)} and \textit{Protein (P)} in the set $\scr{S} = \{ G, M, P\}$. We consider four reactions. 
			\begin{enumerate}
				\item \textbf{Transcription: }$G\to G+M$ with rate constant $\kappa_1$. 
				\item \textbf{Translation: }$M\to M+P$ with rate constant $\kappa_2>0$. 
				\item \textbf{Degradation of mRNA: }$M\to\emptyset$ with rate constant $d_M>0$. 
				\item \textbf{Degradation of protein: }$P\to\emptyset$ with rate constant $d_P>0$.
			\end{enumerate}
Let $X(t)$ be the count vector at time $t$ of species counts $(G, M, P)$. 
			In this example $\eqref{eq:FormOfX}$ translates to 
			\begin{equation} 
				\label{eq:GTT-formofX}
				\begin{aligned}
					X(t) &= x_0 + Y_1\left(\int_0^t \kappa_1 X_{1}(s)\dd s\right) \Bm 0 \\ 1 \\ 0 \Em +  	Y_2\left(\int_0^t \kappa_2 X_{2}(s)\dd s\right) \Bm 0 \\ 0 \\ 1 \Em \\
					& \quad + Y_3\left(\int_0^t d_M X_{2}(s)\dd s\right) \Bm 0 \\ -1 \\ 0 \Em + Y_4\left(\int_0^t d_P X_{3}(s)\dd s\right) \Bm 0 \\ 0 \\ -1 \Em
				\end{aligned},  
			 \end{equation}
		where $Y_1,Y_2,Y_3,Y_4$ are independent unit rate Poisson processes. A realization of this process can be found in \Cref{fig:GTT-forwardsimulation}. 
		
		\begin{figure}[h]
			\centering
			\includegraphics[width = 0.8 \textwidth]{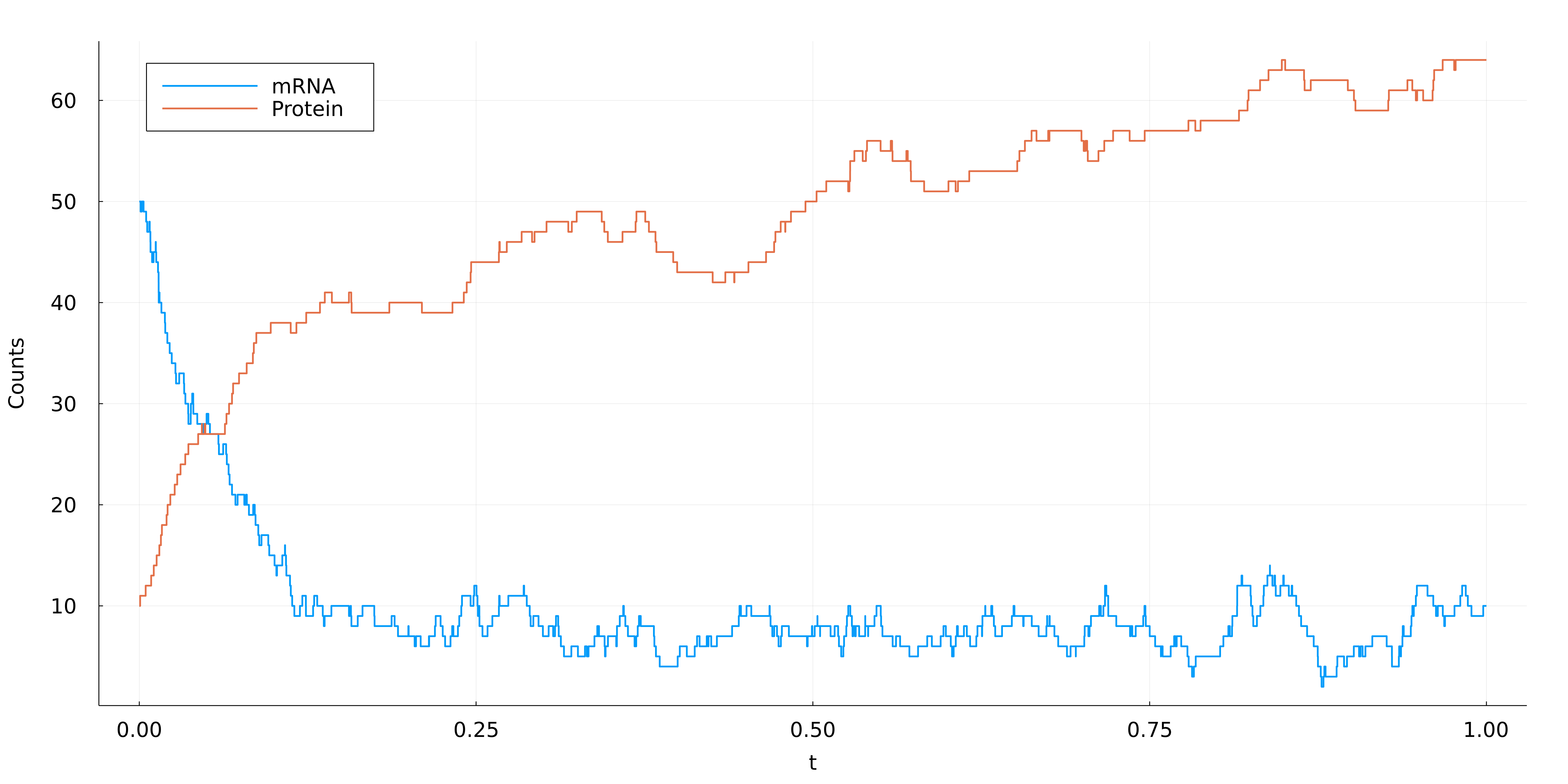}
			\caption{Realization of \eqref{eq:GTT-formofX} using $\kappa_1 = 200$, $\kappa_2 = 10$, $d_M = 25$, $d_P = 1$ and initial position $x_0 = (1, 50, 10)$. Note that the gene count is constant in this process. Therefore it was omitted from the figure.}
			\label{fig:GTT-forwardsimulation}
		\end{figure}
	\end{ex}
	
		\begin{ex}[Enzyme kinetics]
		\label{ex:Enzyme-kinetics}
		The standard model for describing enzyme kinetics, see e.g. \cite{bersani2008}, where a substrate binds an enzyme  reversibly to form an enzyme-substrate complex, which can in turn deteriorate into an enzyme and a product. We thus model  \textit{Substrate (S)}, \textit{Enzyme (E)}, \textit{Enzyme-substrate (SE)} and \textit{Product (P)} in the species set $\scr{S} = \left\{S, E, SE, P\right\}$ and consider of the following reactions. 
		\begin{equation}
			\label{eq:enzyme-kinetics-reactions}
			S+E \overset{\kappa_1}{\underset{\kappa_2}{\rightleftharpoons}} SE \overset{\kappa_3}{\longrightarrow}P+E
		\end{equation}
		Equivalently: 
		\begin{enumerate}
			\item $S+E \to SE$ with rate constant $\kappa_1$. 
			\item $SE \to S+E$ with rate constant $\kappa_2$. 
			\item $SE\to P+E$ with rate constant $\kappa_3$. 
		\end{enumerate}
		Then the reaction rates corresponding to the above listed three reactions are given by 
		\begin{enumerate}
			\item $\lambda_1(x)=\kappa_1 x_1x_2$ and $\xi_1 = (-1,-1,1,0)$. 
			\item $\lambda_2(x) = \kappa_2 x_3$ and $\xi_2 = (1,1,-1,0)$. 
			\item $\lambda_3(x) = \kappa_3 x_3$ and $\xi_3 = (0,1,-1,1)$. 
		\end{enumerate}
							Here $(x_1, x_2, x_3, x_4)$ refer to species counts of $(S, E, SE, P)$. 
		Interesting aspects of this example are firstly that the fourth component ($P$) only appears in reaction ($3$) where $1$ is added and therefore is monotonically increasing and secondly that there are {absorbing} states such as $x= (0, x_2, 0, x_4)$ where the process is killed as all reaction rates are zero. A realization of this process can be found in \Cref{fig:enzyme-kinetics-forwardsimulation}. 
	
	\begin{figure}[h]
		\centering
		\includegraphics[width = 0.8 \textwidth]{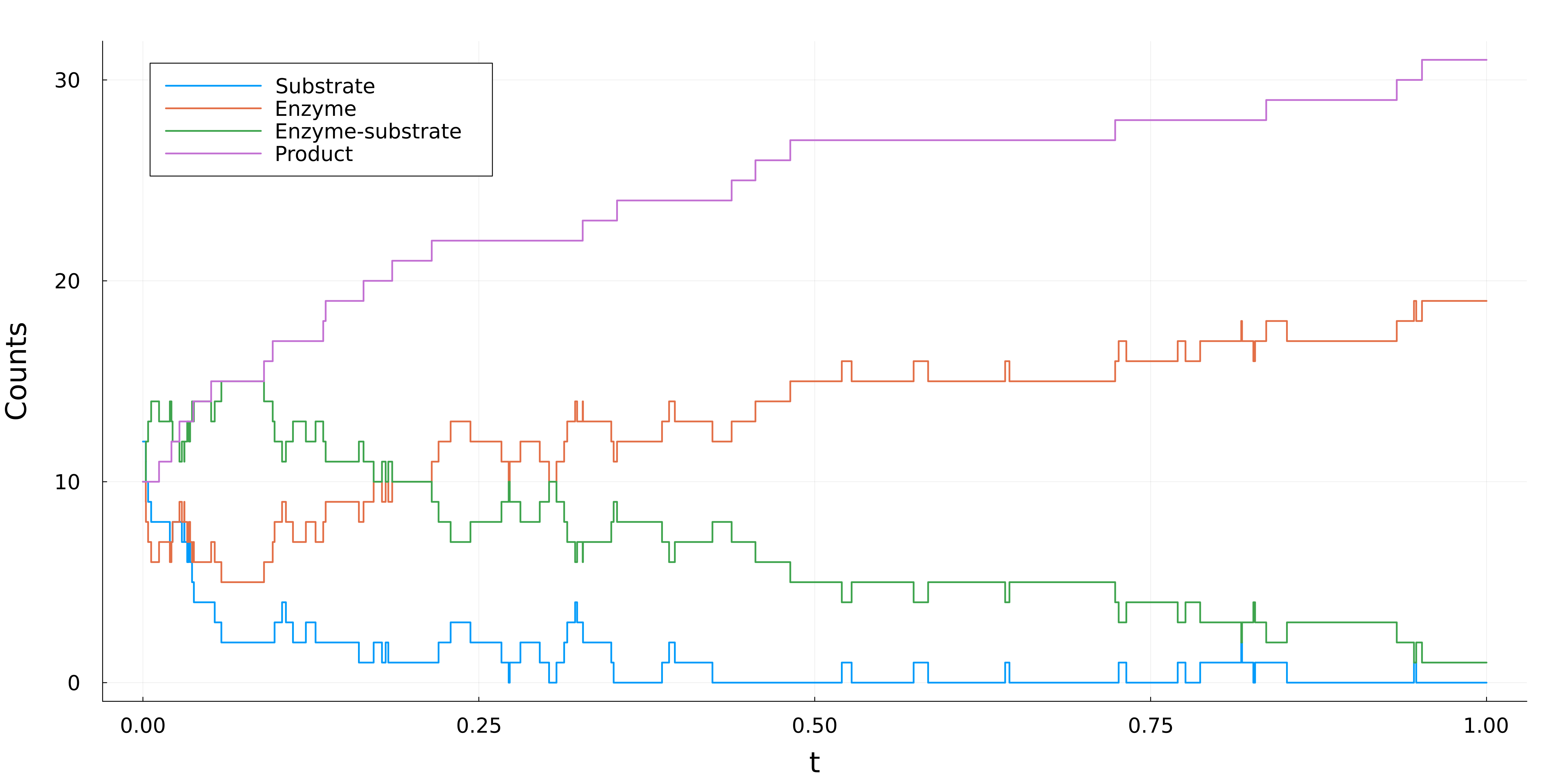}
		\caption{Realization of the forward process of \Cref{ex:Enzyme-kinetics} using $\kappa_1 = 5$, $\kappa_2 = 5$, $\kappa_3=3$ and initial position $x_0 = (12, 10, 10, 10)$. }
		\label{fig:enzyme-kinetics-forwardsimulation}
	\end{figure}
	\end{ex}

	\section{Guided Markov processes}
	\label{sec:doobh}
	
	\subsection{First ideas}
		{We first provide some intuition on how the dynamics of a chemical reaction network change upon conditioning on a future event. Let $\mathcal{E}$ denote  some event later than time $t$, for example $\{X(T)=v\}$. 
		For $t \in [0,T)$, $\Delta >0$ and $x\in \bb{S}$ and $y=x+\xi_\ell$, 
		\[ \begin{aligned}
		&\bb{P}\left(X(t+\Delta)	=y\mid X(t)=x, \mathcal{E}\right) = \frac{\bb{P}\left(X(t+\Delta)=y,\, \mathcal{E} \mid X(t)=x\right)}{\bb{P}\left(\mathcal{E}\mid X(t)=x\right)}\\ 
		&\qquad\qquad \qquad= \frac{\bb{P}\left(\mathcal{E}\mid X(t+\Delta)=y,\,X(t)=x\right)\bb{P}\left(X(t+\Delta)=y\mid X(t)=x\right)}{\bb{P}\left(\mathcal{E}\mid X(t)=x\right)} \\
		&\qquad \qquad\qquad= \bb{P}\left(X(t+\Delta)=y\mid X(t)=x\right) \frac{\bb{P}\left(\mathcal{E}\mid X(t+\Delta)=y\right)}{\bb{P}\left(\mathcal{E}\mid X(t)=x\right)} 
		\end{aligned}\]
		Taking the limit $\Delta \downarrow 0$, this suggests that  if $X$ is conditioned on the event $\mathcal{E}$ then  it  is a chemical reaction process  with adjusted intensities 
		\begin{equation}
		\label{eq:transformed-intensities}
			\lambda_\ell^h(t,x)  = \lambda_\ell(t,x)\frac{h(t,x+\xi_\ell)}{h(t,x)}, 
		\end{equation}
		where  $h(t,x) = \bb{P}\left(\mathcal{E}\mid X(t)=x\right)$. 
				The process with intensities $\lambda_\ell^h(t,x)$ has generator
		\begin{equation}\label{eq:generatorunderchangeofmeasure} 
			\scr{L}^h_t f(x) = \sum_{\ell\in\mathcal{R}} \lambda_\ell^h(t,x)\left(f(x+\xi_\ell)-f(x)\right) .
		\end{equation}
		
		To sample the conditioned process, the function $h$ is required which is rarely available in closed form. The general idea behind the construction of what we call a  guided chemical reaction process is to replace $h$ by a suitable tractable substitute $g$. This gives a process with rates $\lambda_t^g$. Discrepancies between the true conditioned process and the guided process can be accounted for by evaluating the likelihood ratio of their measures on path space. In this section we will establish sufficient conditions on $g$ for this approach to be valid. \\
		
	As shown below, conditioning the process corresponds to a change of measure. To see this connection, a direct computation  shows that  $\scr{L}^h_t f$ as defined in \eqref{eq:generatorunderchangeofmeasure} can be expressed in terms of $\scr{L}_t$:
		\begin{equation}
			\label{eq:Doobh-generator} 
			\mathcal{L}_t^h f(x) = \frac{1}{h(t,x)} \left[ \scr{L}_t (fh)(t,x) - f(x)\scr{L}_th(t,x)\right] .
		\end{equation}
		This operator is strongly connected to exponential changes of measure studied in \cite{palmowski2002} and \cite{corstanje2021conditioning} for a wider class of Markov processes. 

\subsection{Guiding by a change of measure}	
{ We  assume $x_0$ to be known and that {$\mathcal{E}=\{LX(T)=v\}$} is observed {for some known $v\in\bb{R}^m$} and a matrix $L\in\bb{R}^{m\times d}$ of full row rank with $m\leq d$. In case $m=d$, we assume without loss of generality that $L=I$. Additionally, we assume $\mathbb{P}(LX(T)=v \mid X(0)=x_0)>0$. }
 The extension to multiple observations will follow in a straightforward way in later sections.  Let  $\scr{A}$ denote  the infinitesimal generator of the space-time process $(t,X(t))$. That is $(\scr{A}g)(s,x)=\lim_{\Delta\downarrow 0} \Delta^{-1} \EE\left[ g(s+\Delta, X(s+\Delta)) - g(s,X(s)) \mid X(s) = x\right]$, for those $g$ (that map to $\mathbb{R}$) for which the limit exists. Such $g$ are said to be in the domain of the generator, denoted by $\scr{D}(\scr{A})$. While $\scr{D}(\scr{A})$ is implicitly defined, as part of the definition of $\scr{A}$, we can find a more descriptive description on a smaller set of functions. For those functions $\scr{A} = \partial_t + \scr{L}$, where $\scr{L}$ is defined in   
 \eqref{eq:Generator}. From this, one can see that for instance functions $g$ that are differentiable with respect to their time-argument will typically be in the domain.

	For a suitable $g\colon [0,T) \times \bb{S} \to \bb{R}$  define 
	\begin{equation}
		\label{eq:doobh}
		E^g(t) = \frac{g(t, X(t))}{g(0,x_0)}\exp\left( -\int_{0}^t \frac{\scr{A}g}{g}(s, X(s))\dd s \right), \qquad t\in [0,T).
	\end{equation}

\begin{defn}
We write $g\in \scr{G}$ if $g\in\scr{D}(\scr{A})$ is a strictly positive function and there exists a filtration $(\scr{F}_t)_t$ such that $E^g(t),\, t\in [0,T)$ is a martingale adapted to $(\scr{F}_t)_t$. 
\end{defn}

As $(\scr{F}_t^X)_t = \sigma(X(s) : s\leq t)$ is a sub-$\sigma$-algebra of $(\scr{F}_t)_t$,  $E^g(t)$ is a martingale with respect to $(\scr{F}_t^X)_t$ as well whenever $g\in\scr{G}$. In \Cref{lem:martingale-lemma}, we give an explicit condition under which $g\in\scr{G}$. \\

	Let $g\in\scr G$ and let $\bb{P}_t$ denote the law of the process $X$ restricted to $\scr{F}_t^X$. We define a new collection of measures $\{\bb{P}_t^g : t\in[0,T)\}$ by defining $\bb{P}_t$ on $(\Omega,\scr{F}_t^X)$ by 
		\begin{equation}
			\label{eq:DefPstar}
			\dd \bb{P}_t^g =  E^g(t) \dd\bb{P}_t, \qquad t\in [0,T).
		\end{equation}
By {Theorem 4.2} of \cite{palmowski2002}, $f(t,X(t)) - \int_0^t \left( \partial_t f+ \scr{L}^g_t f \right)(s,X(s)) \dd s$ is a martingale under $\bb{P}^g_t$ for  any function $f(t,x)$ finitely supported in  $x$ and differentiable in $t$. Hence, under $\mathbb{P}^g_t$, $X$ is a chemical reaction process with the same jumps as under $\bb{P}_t$ but with adjusted  intensities $\lambda^g_\ell(t,x) = 
\lambda(t,x) g(t,x+\xi_\ell)/g(t,x)$. For the reader's convenience we have summarised some of the main arguments of \cite{palmowski2002} for establishing this connection in Appendix \ref{sec:change of generator}.

\begin{defn}
The \emph{guided} process induced by $g$ on $[0,t]$ is defined as the process $X$ under $\mathbb{P}^g_t$. 
\end{defn}
Some authors refer to this process as the \emph{twisted} process (see e.g.\ \cite{del2017stochastic}). 
Upon denoting the transition probabilities of $X$ by $p$, i.e. $\bb{P}(X(T) \in A\mid X(t) =x) =\sum_{y\in A} p(t,x; T,y)$ for $A\subseteq \bb{S}$, it follows from Example 2.4 of \cite{corstanje2021conditioning} that the conditioned process $\left(X\mid LX(T)=v\right)$ is obtained by taking $g=h$ with 
\begin{equation} 
	\label{eq:defh-single}
	h(t,x) = \sum_{\zeta \in L^{-1} v} p(t,x; T, \zeta),
\end{equation}
where $L^{-1}v = \{ y\in\bb{S} : Ly=v \}$. We recall our assumption $h(0,x_0)>0$, {as otherwise conditioned paths have probability zero}. Note that $h$ is bounded and  satisfies Kolmogorov's backward equation: $\scr{A} h=0$. By  Proposition 3.2 in \cite{palmowski2002}, $h \in \scr{G}$ and thus we can define the probability measure $\bb{P}^h_t$ by 
\[  \dd \bb{P}_t^h = \frac{h(t,X(t))}{h(0,x_0)} \dd \bb{P}_t,\qquad {t \leq T}.\]
Intuitively, $\bb{P}_t^h$ gives more mass to paths where $h(t, X(t))$ is large. 

Unfortunately, $h$ is intractable as the transition probabilities $p$ are only known in closed-form in very specific cases including \Cref{ex:deathprocess}. 

To resolve this, consider another function $g\in \scr{G}$ that acts as a tractable substitute for $h$. Then, as a consequence, the process $X$ is tractable and can be simulated under the measure $\bb{P}_t^{g}$. Moreover, for $t<T$, 
	\begin{equation}		\label{eq:AbsoluteContinuity}
		\der{\bb{P}^h_t}{\bb{P}^{g}_t}(X) = \frac{E_t^h}{E_t^{g}} = \frac{h(t,X(t))}{g(t,X(t))}\frac{g(0,x_0)}{h(0,x_0)}\Psi_t^{g}(X), 
	\end{equation}
	where
	\begin{equation}\label{eq:Psi} \Psi_t^{g}(X) = \exp\left(\int_0^t \frac{\scr{A}g}{g}(s,X(s))\dd s\right).  \end{equation}
To evaluate $\Psi^g$,  a direct computation yields 
\begin{equation} 	
	\label{eq:Ah/h}\frac{\scr{A}g}{g}(s,x) = \partial_s \log g(s,x) + \sum_{\ell\in\scr{R}} \left(\lambda_\ell^{g}(s,x) - \lambda_\ell(s,x)\right). 
\end{equation}
\Cref{eq:AbsoluteContinuity} becomes particularly useful if it can be evaluated in $t=T$ as well. Since $h(T,x) = \ind_{\{Lx=v\}}$ is defined in $T$ this only depends on the choice for $g$. We present a class of functions that can be used for this purpose in \Cref{thm:absolute-continuity-general}.

\begin{lem}
\label{lem:martingale-lemma}
	Suppose $g\in\scr{D}(\scr{A})$ is a strictly positive function such that for some positive constant $C$ 

	\[ \int_0^T \sum_{\ell\in\scr{R}} \lambda_\ell(s,X(s))\left(\frac{g(s,X(s)+\xi_\ell)}{g(s,X(s))}-1\right)^2 \dd s < C,\]
$\mathbb{P}$-almost surely, then 	
	 $g\in\scr{G}$. 
\end{lem}
\begin{proof} The proof is inspired by the proof given in Example 15.2.10 in \cite{bremaud2020probability} and Lemma 19.6
in \cite{liptser2013statistics}.

	It suffices to show that $\bb{E}\left[E^g(T)\right]=1$. By Lemma 3.1 of \cite{palmowski2002}, $E^g$ is a local martingale. Let $\left\{\sigma_n\right\}_n$ be a localizing sequence for $E^g$. Then, if $(E^g(t),\, t\in [0,T])$ were uniformly integrable, then 
	\[ \bb{E}\left[E^g(T)\right] =  \lim_{n\to\infty}\bb{E}\left[E^g(T\wedge\sigma_n)\right] = \bb{E}\left[E^g(0)\right]=1\]
	 We proceed to show uniform integrability. Given $t\in[0,T]$, suppose a trajectory $X$ has jumps at times $t_1,\dots,t_N$.  Set $t_0=0$ and $t_{N+1}=t$. Then 
	 	\begin{equation} 
	\label{eq:afleiding-Eg(T)}
	\begin{aligned}
 	E^g(t) &= \frac{g(t,X(t))}{g(0,x_0)}\exp\left( -\sum_{j=0}^N\int_{t_j}^{t_{j+1}}\partial_s \log g(s, X(t_j))\dd s - \int_0^t \frac{\scr{L}g(s,X(s))}{g(s,X(s))}\dd s \right) \\
 	&= \frac{g(t,X(t))}{g(0,x_0)}\prod_{j=0}^N\frac{g(t_j,X(t_j))}{g(t_{j+1},X(t_j))} \exp\left(-\int_0^t\frac{\scr{L}g(s,X(s))}{g(s,X(s))}\dd s \right) \\
 	&= \prod_{j=0}^{N-1}\frac{g(t_{j+1},X(t_{j+1}))}{g(t_{j+1},X(t_j))} \exp\left( -\int_0^t \frac{\scr{L}g(s,X(s))}{g(s,X(s))}\dd s\right)\\
 	&= \prod_{j=0}^{N-1}\frac{g(t_{j+1},X(t_{j+1}))}{g(t_{j+1},X(t_j))} \exp\left( -\int_0^t \sum_{\ell\in\scr{R}} \lambda_\ell(s,X(s))\left(\frac{g(s,X(s)+\xi_\ell)}{g(s,X(s))}-1\right) \dd s\right)
 \end{aligned}
 \end{equation}
From \eqref{eq:afleiding-Eg(T)}, it is easy to verify that 
\[ E^g(t)^2 = E^{g^2}(t)\exp\left(  \int_0^t \sum_{\ell\in\scr{R}} \lambda_\ell(s,X(s))\left(\frac{g(s,X(s)+\xi_\ell)}{g(s,X(s))} -1 \right)^2 \dd s \right) \]
By Lemma 3.1 of \cite{palmowski2002}, $E^{g^2}$ is a local martingale and, since it is bounded from below, it is a supermartingale with $\bb{E}\left[E^{g^2}(t)\right]\leq 1$. Hence, for all $n$,
\[\begin{aligned} 
 &\mathbb{E}\left[E^g(T\wedge\sigma_n)^2\right]  \\
 &\quad \leq \bb{E}\left[ E^{g^2}(T\wedge\sigma_n)\exp\left(  \int_0^T \sum_{\ell\in\scr{R}} \lambda_\ell(s,X(s))\left(\frac{g(s,X(s)+\xi_\ell)}{g(s,X(s))} -1 \right)^2 \dd s \right)\right]	
  \leq e^C
\end{aligned}
\]
Therefore $\sup_n \mathbb{E}\left[E^g(T\wedge\sigma_n)^2\right] <\infty$. 
\end{proof}

\begin{thm}
	\label{thm:absolute-continuity-general}
	Define $h$ is as in \eqref{eq:defh-single} and let $g\colon[0,T]\times\bb{S}\to\bb{R}$ be such that the condition of \Cref{lem:martingale-lemma} is satisfied. Then 
	\begin{equation}\label{eq:radon-nikodym-abstract} 
		\frac{\dd\bb{P}^h_T}{\dd\bb{P}^g_T} = \frac{g(0,x_0)}{h(0,x_0)}\frac{\Psi_T^g(X)}{g(T,X(T)}\ind_{\{LX(T)=v\}}. 
	\end{equation} 
\end{thm}
\begin{proof}
	The form of the Radon-Nikodym derivative follows from \eqref{eq:AbsoluteContinuity} upon noting that for $t\uparrow T$, 
\[
	 h(t,X(t)) = \bb{E}\left[ \ind\{ LX(T)=v\} \mid \scr{F}_t\right] \to \bb{E}\left[ \ind\{ LX(T)=v\} \mid \scr{F}_T\right] = \ind\{LX(T)=v\}.
\]
\end{proof}
The following proposition is often helpful in establishing that the condition of \Cref{lem:martingale-lemma} is satisfied for a given $g$. 

\begin{prop}
\label{prop:upper-lower-bounds-g}
	Suppose  $g\colon[0,T]\times\bb{S}\to\bb{R}$ is such that for all $x$ the map $t\mapsto g(t,x)$ is bounded from above and bounded away from zero. Then the condition of \Cref{lem:martingale-lemma} is satisfied.
\end{prop}
\begin{proof}
	Since we assume throughout $X$ to be nonexplosive, we have that almost surely the maps $t\mapsto g(t,X(t)+\xi_\ell)/g(t,X(t))$ are almost surely bounded on $[0,T]$ for all $\ell\in\scr{R}$. The result now follows upon noting that $\int_0^T \sum_{\ell\in\scr{R}} \lambda_\ell(s,X(s))\dd s <\infty$ by Assumption \ref{ass:AssumptionsOnProcess-FiniteSum}.  
\end{proof}

\begin{cor}
\label{corr:likelihood}
The likelihood of $x_0$ based on the observation $v$ is given by
\begin{equation}
\label{eq:likelihood_abstract}
	h(0,x_0) = g(0,x_0) \bb{E}_T^{g}\left[ \frac{\Psi_T^{{g}}(X)}{g(T,X(T)}\ind_{\{LX(T)=v\}}\right]. 
\end{equation}

\end{cor}
\begin{proof}
	This follows immediately upon integrating \eqref{eq:radon-nikodym-abstract} with respect to $\bb{P}^{g}$. 
\end{proof}

Note that  $g\equiv 1$ satisfies the condition in \Cref{lem:martingale-lemma}.  This choice simply yields the original forward process. Sampling the conditioned process in this way  is however very inefficient when $\bb{P}\left(LX(T)=v\right)$ is low. \\

The lemma below shows that the law of the guided process does not change when $g$ is multiplied by a function only depending on time. 
\begin{lem}[Invariance under time scaling]
	\label{lem:time-scaling}  
	Suppose $g\in\scr{G}$ and let $c\colon [0,T]\to\bb{R}_+$ be a differentiable function. Then $E^{cg}(t) = E^{g}(t)$.
\end{lem}
\begin{proof}
	Since $c$ only depends on time
	\[ \begin{aligned} 
		\int_0^t \frac{\scr{A}\left(cg\right) }{cg}(s,X(s))\dd s &= \int_0^t \left[ \partial_s\log\left( c(s)g(s,X(s))\right) + \frac{\scr{L}_s\left(cg\right)}{cg}(s,X(s)) \right]  \dd s \\  
		&= \int_0^t \partial_s \log c(s)\dd s + \int_0^t \left[ \partial_s\log g(s,X(s)) + \frac{\scr{L}_sg}{g}(s,X(s))\right] \dd s \\ 
		&= \log c(t)- \log c(0) + \int_0^t \frac{\scr{A}g}{g}(s,X(s)) \dd s.
	\end{aligned}. \] From this, we get $E^{cg}(t)$ by negating the right-hand-side, taking the exponent and subsequently multiplying by $c(t) g(t,X(t)) /(c(0)g(0,x_0))$. It is easily seen that the terms with $c$ cancel out.
\end{proof}
The following proposition is useful for numerical evaluation of the likelihood of a sampled guided process. 
\begin{prop}
Suppose $0=t_0 < t_1 < \cdots <t_{N-1}<t_N=T$, $x_0,\dots, x_N\in\bb{S}$ and 
\[ x(t) = \sum_{j=0}^{N-1} x_j \ind_{[t_j,t_{j+1})}(t). \]
Then, with $\alpha_j(s):=g(s,x_{j+1})/g(s,x_j)$
\begin{equation*}
\begin{split} 
& \frac{g(0,x_0)}{g(T, x(T))} \exp\left(\int_0^T \frac{\scr{A}g}{g}(s,x(s))\dd s \right) \\ &\qquad 
=  \left( \prod_{j=0}^{N-2} \frac{1}{\alpha_j(t_{j+1})} \right)\exp\left( \sum_{j=0}^{N-1} \int_{t_j}^{t_{j+1}}  \sum_{\ell \in\scr{R}} \lambda_\ell(s, x_j)\left[\alpha_j(s) - 1\right] \dd s
\right)
\end{split}
\end{equation*}

\end{prop}
\begin{proof}
This is  a consequence of \eqref{eq:afleiding-Eg(T)}.	
\end{proof}

\section{Choices for $g$}

\label{sec:choices_for_g}
\label{subsec:guiding_term_scaled_diffusion}
It remains to specify the maps $(t,x) \mapsto g(t,x)$ satisfying the assumption of \Cref{lem:martingale-lemma}. Moreover, to be of any use,  the event $\{LX(T)=v\}$ needs to get positive probability under $\bb{P}^g_T$.

In this section we show the following:
				\begin{itemize}
			\item Section \ref{subsec:comparison}: The choices for $g$ in \cite{fearnhead2008computational} and \cite{golightly2019efficient}, both based on the chemical Langevin equation, lead to absolute continuity. 
			\item Section \ref{subsec:scaled}: A  guiding function based on the transition density of a scaled  Brownian motion  yields absolute continuity while being computationally very efficient. It satisfies $\bb{P}^g_T(LX(T)=v)>0$. In subsection \ref{subsubsec:multiple} we show how $g$ can be extended to the case of multiple partial observations.
			\item Section \ref{sec:poisson}: For processes with monotone components a specific choice for $g$ can overcome certain problems arising in the aforementioned choices.
				\end{itemize}
			 	In the following, $C$ denotes a positive definite matrix.

\subsection{Choices based on the chemical Langevin equation}
\label{subsec:comparison}
Consider the CLE given in  \eqref{eq:CLE}. Set  $a_{\mathrm{CLE}}={\sigma_{\mathrm{CLE}}}{\sigma_{\mathrm{CLE}}}^\T$. 
 \cite{fearnhead2008computational} proposed to choose $g$ based on the  Euler discretization of the CLE. This gives 
\begin{equation}\label{eq:g_fearnhead}
 g_{\rm F}(t,x) = \scr{N}\left(v ; L(x+b_{\mathrm{CLE}}(t,x)(T-t)) , La_{\mathrm{CLE}}(t,x)L^\T(T-t) + C\right),	
\end{equation}
As the assumption of \Cref{prop:upper-lower-bounds-g} is satisfied we obtain  $\bb{P}^h_T \ll \bb{P}^{g_{\mathrm{F}}}_T$.\\

\cite{golightly2019efficient} propose to infer $g$ by using the Linear Noise Approximation (LNA) to the CLE. Here, we discuss their specific choice  called ``LNA with restart''. 
	For notational convenience, write $b=b_{\mathrm{CLE}}$ and $\sigma=\sigma_{\mathrm{CLE}}$ {and again set $a=\sigma\sigma^\T$}. 	
 Given $t<T$ and $x\in\bb{S}$, we denote by $z_{T\mid (t,x)}$ and $V_{T\mid (t,x)}$ the solutions at time $T$ of the system of ordinary differential equations 	 \begin{equation*}
	 	\begin{aligned}
\dd z_{s\mid (t,x)} &= {b}(s, z_{s\mid (t,x)}) \dd s \\
\dd V_{s\mid (t,x)} &=\left( V_{s\mid(t,x)}\left( J_{b}(s, z_{s\mid(t,x)})\right)^\T  + J_{b}\left(s,z_{s\mid(t,x)}\right) V_{s\mid(t,x)}+{a}\left(s, z_{s\mid(t,x)}\right)\right) \dd s 
	 	\end{aligned}
	 \end{equation*} 
where $s \in [t,T]$,	 subject to the initial conditions are given by $z_{t\mid (t,x)} = x$ and $V_{t\mid(t,x)} = 0$.  
	Here, $J_{b}$ denotes the Jacobian matrix of ${b}$ (having component $(i,j)$ given by $\partial b_i / \partial z_j$. The guiding term then is given by 
	\begin{equation}
	\label{eq:guiding_term_LNA} 
	g_{\mathrm{LNAR}}(t,x) = \scr{N}\left( v ; L z_{T\mid (t,x)} , LV_{T\mid (t,x)}L^\T +C\right). 
	\end{equation}
	Absolute continuity  $\mathbb{P}^h_T \ll \mathbb{P}^{g_{\mathrm{LNAR}}}_T$  follows by the same argument used for $g_{\mathrm{F}}$.  
\subsection{Choosing $g$ using the transition density of a scaled Brownian motion}
\label{subsec:scaled}
 
 For fixed $\sigma \in \bb{R}^{d\times d}$ define the process $\tilde{X}$ by 
 $\dd \tilde{X}(t) = {\sigma}\dd W(t)$, where  $W$ is a standard Brownian motion. Denote ${a}={\sigma}{\sigma}^\T$ and assume that ${a}$ is such that $L{a}L^\T$ is strictly positive definite. 
  We derive  $g$ from backward filtering the process $\tilde{X}$ with observation $V\mid \tilde{X}(T) \sim \scr{N}(L\tilde{X}(T),C)$. Let $q$ denote the density of the $\scr{N}(0,C)$-distribution. The density of $V$, conditional on $\tilde{X}(t)=x$ is given by $u(t,x):=\int \tilde{p}(t,x;T,y)q(v-Ly)\dd y$. It follows from the results in \cite{mider2020continuousdiscrete} that $u(t,x)\propto g(t,x)$ where
\begin{equation}\label{eq:ourchoice_g}
	 g(t,x)= \exp\left(-\frac12 x^\T H(t) x + F(t)^\T x\right) .
\end{equation}

Here, for $t\leq T$, $H(t)$ and $F(t)$ satisfy the ordinary differential equations
  \begin{equation}\label{eq:HFsimple}
  	\begin{aligned} \dd H(t) &= H(t){a} H(t) \dd t, \qquad H(T) = L^\T C^{-1} L  \\
  		\dd F(t) &= H(t) {a} F(t)\dd t,\qquad F(T)=L^\T C^{-1}v\end{aligned}. 
  	\end{equation}
 
These differential equations can be solved in closed-form:
\begin{equation}
	\label{eq:HF-closedform}
	H(t) = z(t)H(T) \quad \text{and} \quad
	F(t) = z(t) F(T),
\end{equation}
with $z(t)=\left(I+H(T){a}(T-t)\right)^{-1}$

Note that we have not yet specified $\sigma$. Depending on its choice,  $g$ may be radically different from $h$. Nevertheless,  it can be used as a change of measure to condition paths of $X$ on the event $\{LX(T)=v\}$.

	\begin{thm}\label{thm:main-single-obs}
		Let $g$ be defined by \eqref{eq:ourchoice_g} and  \eqref{eq:HFsimple}. Then 	$\bb{P}_T^h\ll\bb{P}_T^{g}$ with 
		\begin{equation}  
		\label{eq:radon-nikodym-1obs}
			\der{\bb{P}_T^h}{\bb{P}_T^{g}} = \frac{g(0,x_0)}{h(0,x_0)} \exp\left( \int_0^T \frac{\scr{A}g}{g} (s,X(s))\dd s \right)\exp\left(-\frac12 v^\T C^{-1}v\right) \ind_{\{LX(T)=v\}}.
		\end{equation}

	\end{thm}
	\begin{proof}
	It follows from theorems 2.4 and 2.5 in \cite{mider2020continuousdiscrete} that if we define
 	\begin{equation*}
 		M(t) := \left(C+L{a}L^\T (T-t)\right)^{-1},
 	\end{equation*}
	then $H(t) = L^\T M(t)L$ and $F(t) =L^\T M(t)v$ and therefore
 	\begin{equation}
 		\label{eq:htilde-noHF}
 		g(t,x) \propto \exp\left( -\frac12 (v-Lx)^\T M(t)(v-Lx) \right).
 	\end{equation}
	The proportionality is up to a differentiable time-dependent function and hence does not affect the change of measure by \Cref{lem:time-scaling}.
	By \Cref{lem:Bounds-on-htilde}, $g$ is bounded from above and bounded away from $0$. Since $g$ is  smooth in $t$ and well-defined in $T$, we clearly have continuous differentiability of the maps $t\mapsto g(t,x)$ for all $x\in\bb{S}$ on $[0,T]$. The result thus follows from \Cref{prop:upper-lower-bounds-g} and \Cref{thm:absolute-continuity-general}.  
	\end{proof}

\Cref{thm:main-single-obs} is only of interest  when the guided process has a positive probability of hitting $v$ at time $T$. This is ensured by \Cref{thm:positive-prop-1obs} for which the proof is deferred to \Cref{app:proof-of-path-exists}.  

\begin{thm}\label{thm:positive-prop-1obs} $\bb{P}^{g}(LX(T)=v)>0$ 
\end{thm}

A convenient choice for the matrix $C$ has computational advantages. 
\begin{defn}
\label{def:g_eps}
	Let $\eps>0$ and set $C = \eps L{a}L^\T$. The corresponding $g$ in \eqref{eq:ourchoice_g} is denoted by $g_\eps$.
\end{defn}

\begin{lem}\label{lem:g_epsilon}
The induced measure $\mathbb{P}^{g_\eps}_T$ is the law of a chemical reaction process with intensities $\lambda_\ell^{g_\eps}(t,x)=\alpha^{g_\eps}_\ell(t,x) \lambda_\ell(t,x)$, where 
\begin{equation}
	\label{eq:guided_rate_worked_out}
\alpha^{g_\eps}_\ell(t,x) = \exp\left( - \frac{ d(v, L(x+\xi_\ell))^2-d(v,Lx)^2}{2(\eps + T-t )} \right), 
\end{equation}
with $d$ the metric given by 
  \begin{equation}
 	\label{eq:metric-1obs}
	d(x,y)^2 = (y-x)^\T(LaL^\T)^{-1}(y-x).
 \end{equation}
\end{lem}
\begin{proof}
 	This result follows by substituting the expression for $g$ in \eqref{eq:htilde-noHF} into the definition of $\alpha^g_\ell(t,x)$ and rewriting in terms of the metric $d$. 
\end{proof}
This implies that only one matrix inverse, $(LaL^\T)^{-1}$,  needs to be calculated.  
By \Cref{eq:guided_rate_worked_out}, the intensity of the guided process either becomes very large or small depending on whether $L(x+\xi_\ell)$ is closer or further away from $v$ than $Lx$ with respect to the metric $d$, which has intuitive appeal. Also note that $\eps$ appears only in the denominator at $T+\eps$, which implies that the choice $C=\eps L{a}L^\T$ {imposes the conditioning at at time  $T+\eps$ instead of $T$}, thereby precluding explosive behaviour in \eqref{eq:htilde-noHF} as $t\uparrow T$.

\begin{rem}
	\label{rem:prob-isnot1}
	Ideally, we would like to have $\bb{P}^{g}\left(LX(T)=v\right)=1$. For $g=g_\eps$ as in \Cref{lem:g_epsilon} this is not guaranteed. Essentially, the choice $C=\eps L{a}L^\T$ yields a process where $LX(t)$ hits $v$ at time $T+\eps$. Since the intensities are bounded from below and above, there is a positive probability of at least one reaction occurring in $(T,T+\eps)$ resulting in $LX(T)\neq v$. In \Cref{sec:equivalence}  we discuss the case where $\eps=0$, which does yield $\bb{P}^{g}(LX(T)=v)=1$. 
\end{rem}

\begin{rem}
Comparing $g_{\mathrm{F}}$, $g_{\mathrm{LNAR}}$ and $g_\eps$, it is 
{likely} that $g_\eps$ will deviate most from $h$. In particular when the drift and diffusion coefficient of the CLE are nonlinear, $g_\eps$ will likely deviate from $h$ the most leading to the guided paths deviating from the conditioned paths. However, from a computational point of view  $g_\eps$ is most attractive, as can be seen from \eqref{eq:guided_rate_worked_out}.  Moreover, in Section \ref{sec:simulations} it turns out that simulation of the guided process is simplified in case of $g_\eps$. 
\end{rem}

\begin{rem}\label{rem:choice_a}
	In dimension $1$, we derive a better intuition for the choice for $a$ from \eqref{eq:guided_rate_worked_out} and \eqref{eq:metric-1obs}. A small choice leads to a higher guided intensity, and thus a process arriving around $v$ quickly and likely to stay near $v$, while a large value of $a$ leads to trajectories that remain unaffected by $\alpha_\ell^{g_\eps}$ until $t$ approaches $T$. Ideally, we mimic trajectories under $\bb{P}^h$ and thus a good choice of $a$ ensures $g(0,x_0)\Psi_T^g(X)\ind\{LX(T)=v\}/g(T,X(T))$ is large (with $\Psi_t^g(X)$ defined in Equation \eqref{eq:Psi}). This can be verified using Monte-Carlo simulation sampling  $X$ under $\bb{P}^{g_\eps}$.  Alternatively, a simple choice for $a$ is  $a_{\mathrm{CLE}}(0,x_0)$ or, if we have a complete observation, $a=a_{\mathrm{CLE}}(T,v)$. 
\end{rem}

\subsubsection{Extension to multiple observations}\label{subsubsec:multiple}

	We now extend \Cref{thm:main-single-obs} to a result for multiple observations. Consider observations $v_i = L_i X(t_i)$, $i=1,\dots,n$ where $0=t_0<t_1<\cdots<t_n$ and assume without loss of generality that $L_i \in\bb{R}^{m_i\times d}$ are of full column rank with $m_i\leq d$ and $L_i = I$ when $m_i=d$. \\

	\begin{prop}
		\label{prop:ConditionedProcess}
		Let $p$ denote the transition probabilities of $X$ and define for $t\in[t_{k-1},t_k)$ and $x\in\bb{S}$
		\begin{equation}
			\label{eq:defh-multiple}
			\begin{aligned}
			h(t,x) &= \bb{P}\left( L_iX(t_i)=v_i,\, i=k,\dots,n\mid X(t)=x\right) \\
			&= \sum_{\zeta_k \in L_k^{-1}v_k} \cdots \sum_{\zeta_n \in L_n^{-1}v_n} p(t,x;t_k, \zeta_k)\prod_{i=k}^{n-1} p(t_i,\zeta_i;t_{i+1},\zeta_{i+1}).
			\end{aligned}
		\end{equation}
		Then $h\in\scr{G}$ and the change of measure \eqref{eq:DefPstar} induces $\left(X\mid L_kX(t_k)=v_k,\, k=1,\dots,n\right)$. 
	\end{prop}
\begin{proof}
	This result is obtained upon following Example 2.4 of \cite{corstanje2021conditioning} using the delta-dirac distribution $\mu(\zeta_k) = \delta(v_k-L_k\zeta_k)$. 
	\end{proof}

We deduce the form of $g$ from \cite{mider2020continuousdiscrete} in a similar way compared to a single observation. We consider an auxiliary process $\tilde{X}$ that solves the SDE 
 \begin{equation}
 	\label{eq:auxiliary-process}
 	 \dd \tilde{X}(t) = {\sigma}(t) \dd W(t), \qquad X(0)=x_0,\, t\in[0,t_n),
 \end{equation}
 where ${\sigma}(t) = \sum_{k=1}^{n}{\sigma}_k\ind_{[t_{k-1},t_k)}(t)$ and ${a}_k = {\sigma}_k{\sigma}_k^\T$ are positive definite $d\times d$ matrices for $k=1,\dots,n$. 
 
 For each observation $k$, we consider $V_k\mid X(t_k)\sim \scr{N}(0,C_k)$ where $C_k$ is an $m_k\times m_k$ covariance matrix. Suppose $q_k$ {denotes} the density of the $\scr{N}(0,C_k)$-distribution. Then the transition density of $\tilde{X}$ satisfies for $t\in[t_{k-1},t_k)$, 
 
\begin{equation}
	\label{eq:htilde-multiple}
 \begin{gathered}  	\int \tilde{p}(t,x;t_k,\zeta_k)q_k(v_k-L_k\zeta_k)\prod_{i=k}^{n-1}\tilde{p}(t_i,\zeta_i;t_{i+1},\zeta_{i+1})q_{i+1}(v_{i+1}-L_{i+1}\zeta_{i+1}) \dd\zeta_k\cdots\dd\zeta_n  \\ \propto \exp\left(-\frac12 x^\T H(t)x + F(t)^\T x\right)=: g(t,x). 
 \end{gathered}
\end{equation}

Here, the expressions for $H$ and $F$ can be found by backward solving a system of equations given by
\[ \begin{aligned}
 H(t_n)  &= L_n^\T C_n^{-1}L_n \\
 F(t_n) &= 	L_n^\T C_n^{-1}v_n
 \end{aligned} \]
 and for $t\in(t_{k-1},t_k)$, $k=1,\dots,n$
\begin{equation}
\label{eq:HF-ODEs-multiple}
	\begin{aligned}
	\dd H(t) &= H(t)a_k H(t) \dd t, \qquad  H(t_k) = H_k := L_k^\T C_k^{-1}L_k + H(t_k+) \\
 	\dd F(t) &= H(t)a_k F(t) \dd t, \qquad  F(t_k) = F_k := L_k^\T C_k^{-1}v_k + F(t_k+)
	\end{aligned}
\end{equation}
Finally, we let $H$ and $F$ be a right continuous modification of the solution to \eqref{eq:HF-ODEs-multiple}, i.e. setting $H(t_k) = H(t_k+)$ and $F(t_k)=F(t_k+)$, $k=1,\dots,n-1$. 
This system can be solved in closed form: ${H}(t) = z_k(t)H_k$, ${F}(t) = z_k(t)F_k$, where, for $k=1,\dots,n$,  $z_k(t) = \left(I+H_ka_k(t_k-t)\right)^{-1}$.

\begin{rem}
	\label{rem:htilde-multiple-alternative}
	Alternatively, following the computations in Section 2 of \cite{mider2020continuousdiscrete},  $g$ can be expressed as 
	\begin{equation} 
		\label{eq:htilde-multiple-alternative}
		g(t,x) \propto \exp\left( -\frac12 \left(v(t)-L(t)x\right)^\T M(t)\left(v(t)-L(t)x\right) \right), 
	\end{equation}
	where, the proportionality is up to a time-dependent function and for $t\in[t_{k-1},t_k)$
	\[ L(t) = \Bm L_k \\ \vdots \\ L_n \Em  \qquad \text{and}\qquad v(t) = \Bm v_k \\ \vdots \\ v_n \Em \]
	and $M(t) = M^\dagger(t)^{-1}$ with $M^\dagger(t)$ a block matrix with entries
	\[ M^\dagger(t) = \Bm C_i\ind\{i=j\} +\sum_{l=k}^{i\wedge j-1} L_i a_{l+1} L_j^\T(t_{l+1}-t_{l}) + L_i a_{k}L_j^\T (t_{k}-t)\Em_{i,j=k}^n,  \]
$t\in[t_{k-1},t_k)$.
	While this representation of $g$ is useful in most proofs, it is computationally more demanding as the matrix dimensions of $H(t)$ and $F(t)$ are $d\times d$, while $M(t)$, $L(t)$ and $v(t)$ have dimensions that increase with the amount of observations. 
\end{rem}

Using \Cref{lem:bounds-on-htilde-multiple}, the proof of \Cref{thm:main-single-obs} can be repeated for a result like \Cref{thm:main-single-obs} for each observation $k$ to find that $\bb{P}_{t_n}^h \ll \bb{P}_{t_n}^{g}$ with 
\begin{equation}
	\label{eq:Radon-Nikodym-multiple}
	\der{\bb{P}_{t_n}^h}{\bb{P}_{t_n}^{g}} = \frac{g(0,x_0)}{h(0,x_0)}   \exp\left(\int_0^{t_n}\frac{\scr{A}g}{g}(s,X(s))\dd s - \frac12 \sum_{k=1}^n v_k^\T C_k^{-1}v_k\right) \ind_{A_n} .
\end{equation}
The term $\exp\left(- \frac12 \sum_{k=1}^n v_k^\T C_k^{-1}v_k\right)$ results from evaluating each fraction 
$g(t_k, x)/g(t_k-,x)$ on the set $\left\{L_k x=v_k\right\}$. We can repeat \Cref{thm:positive-prop-1obs} for each observation to obtain $\bb{P}^{g}\left(A_n \right) >0$. However, similar to  \Cref{rem:prob-isnot1}, we generally do not have $\bb{P}^{g}\left(A_n \right) = 1 $. 

	\subsection{Choosing $g$ to guide a process with monotone components}
		\label{sec:poisson}
		\Cref{thm:main-single-obs} shows that choosing $g$ using the transition density of a scaled Brownian motion yields a process that is absolutely continuous with respect to the conditioned process. However, depending on the network, this might not always be a desirable choice. \\
		
		Consider for instance \Cref{ex:Enzyme-kinetics}. Here, the fourth component of the process (\textit{P}) is monotonically increasing as it only appears in reaction $3$, where $1$ is added, which can also be seen in \Cref{fig:enzyme-kinetics-forwardsimulation}. Let us for simplicity assume a complete observation at time $T$, that is $X(T)=v_T$. Clearly, if $X_4(t)=v_{T,4}$ for some $t<T$, reaction $3$ cannot occur anymore for the process to satisfy the conditioning. However,  $\lambda_3^{g}(t,X(t)) \neq 0$ for the choices of  $g$ discussed so far. This can lead to trajectories that don't satisfy the conditioning, which in turn means a low acceptance ratio when sampling. \\
		
		Alternatively, we choose an auxiliary process $\tilde{X} = ( \tilde{Z},  \tilde{Y})$, where $\tilde{Z}$ is the $\bb{R}^3$-valued process that solves 
$\dd\tilde{Z}(t) = \sigma \dd W_t$ 
		and $\tilde{Y}$ is a homogeneous Poisson process with intensity $\tilde{\theta}$. We stick with the notation $x = (z,y)\in\bb{R}^4$ with $z\in\bb{R}^3$ and $y\in\bb{R}$, and denote $v_T = (z_T,y_T)$. Similar to earlier computations, we deduce that  
		\begin{equation}
		\label{eq:htilde-poisson}	
		g(t,(z,y)) = \exp\left( -\frac{d(z_T,z)^2}{2(T +\epsilon-t)} \right) \frac{\left(\tilde{\theta}(T-t)\right)^{y_T-y}}{(y_T-y)!}\exp\left(-\tilde{\theta}(T-t)\right), 
		\end{equation}
		where $\epsilon>0$ and 
		\[ d(z_T,z) = \sqrt{(z_T-z)^\T a^{-1}(z_T-z)}. \]
		This choice satisfies the assumption of \Cref{lem:martingale-lemma}. 

\section{Choosing $g$ so that $\bb{P}^h$ and $\bb{P}^{g}$ are equivalent}
\label{sec:equivalence}

By  \eqref{eq:Radon-Nikodym-multiple},  the measures $\bb{P}^h$ and $\bb{P}^{g}$ are equivalent when $\bb{P}^{g}\left( A_n\right) = 1$, as all other terms in the Radon-Nikodym derivative are nonzero.  In this section we propose a choice for $g$ which
 yields equivalence of the measures $\bb{P}^h$ and $\bb{P}^{g}$ under an additional assumption on the network. 

	By \Cref{rem:prob-isnot1}, the choice $C = \eps LaL^\T$ can be interpreted as imposing a condition of hiting $v$ at time $T+\eps$. Therefore, intuitively, equivalence can be achieved upon artificially setting all elements in the matrices $C_k$ equal to $0$. In this case, $g(t_k,x)$ becomes ill-defined whenever $L_kx\neq v_k$ and boundedness of $g$ is lost, rendering the earlier proofs invalid. To utilize this choice for $g$ and show equivalence, we build on earlier work in \cite{corstanje2021conditioning}. Proofs of results in this section are given in \Cref{app:proof-of-main}.

\subsection{Absolute continuity}
	The initial conditions for the differential equations \eqref{eq:HF-ODEs-multiple} for the functions $H$ and $F$ appearing in  \eqref{eq:htilde-multiple} are not defined as the matrices $C_1,\dots,C_n$ are not invertible, but we can still obtain an equivalent of \eqref{eq:htilde-multiple-alternative}. Observe that for $t\in[t_{k-1},t_k)$,
\begin{equation}
\label{eq:htilde-equivalence}
\begin{gathered} 
	\int \tilde{p}(t,x;t_k,\zeta_k)\delta_{v_k}(L_k\zeta_k) \prod_{i=k}^{n-1} \tilde{p}(t_i,\zeta_i;t_{i+1},\zeta_{i+1})\delta_{v_{i+1}}(L_{i+1}\zeta_{i+1}) \dd\zeta_k \cdots \dd\zeta_n \\
	\propto \exp\left(- \frac12 (v(t)-L(t)x)^\T M(t)(v(t)-L(t)x)\right) =: g(t,x),
\end{gathered}
\end{equation}
where $L$, $v$ and $M$ are defined in \Cref{rem:htilde-multiple-alternative} but with $C_1,\dots ,C_n=0$ and we define $g(t_k,x) = g(t_k+,x)$ for $x\in L_k^{-1}v_k$.  

\begin{thm}
	\label{lem:htilde-property}
	Let $g$ be defined by \eqref{eq:htilde-equivalence}. Then $\bb{P}^h \ll \bb{P}^{g}$ with 	\begin{equation}
		\label{eq:absolutecontinuity-czero}
		 \der{\bb{P}_{t_n}^h}{\bb{P}_{t_n}^{g}} = \frac{g(0,x_0)}{h(0,x_0)} \exp\left( \int_0^{t_n} \frac{\scr{A}g}{g}(s,X(s))\dd s\right) \ind_{A[X^{t_n}]}  ,
		\end{equation} 
	where 
	\[	A[X^{t_n}] = \left\{ \sup_{0\leq s <t_n} \sum_{\ell\in\scr{R}} \lambda_{\ell}^{g}(s,X(s)) <\infty \right\}. \] 
\end{thm}

To show that $\bb{P}^{g}_{t_n}(A[X^{t_n}])>0$, we require \Cref{prop:sets}, combined with the observation that the proof of \Cref{thm:positive-prop-1obs} can be repeated with this choice for $g$ upon observing that $\sum_{\ell\in\scr{R}}\lambda_\ell^{g}$ stays finite on the sets $\{L_kx=v_k\}$.
\begin{prop}
	\label{prop:sets}
	$A_n \subseteq A[X^{t_n}]  $ {with $A_n$ as defined in \eqref{eq:defAn}.}
\end{prop}

\begin{rem}
To see that the reverse inclusion generally does not hold, consider a process conditioned to hit $v$ at time $T$. If $LX(t)=u\neq v$ for $t\in(T-\eps,T]$ where $u$ is such that no reactions exist such that $d(v,L(u+\xi_\ell)) < d(v,Lu)$, it can be shown that $t\mapsto \frac{g(t,u+\xi_\ell)}{g(t,u)}$ is bounded and therefore such trajectories are included in $A[X^T]$. 
\end{rem}

		\subsection{Equivalence}
		\Cref{lem:htilde-property} shows that $\bb{P}^h$ is absolutely continuous with respect to $\bb{P}^{g}$ on the set $A[X^{t_n}]$. For simulation purposes, equivalence would be preferable. It follows from \eqref{eq:absolutecontinuity-czero} that this is indeed the case if $\bb{P}^{g}\left(A[X^{t_n}]\right)=1$. This can be shown if the network also satisfies a greedy property.

		\begin{ass}
			\label{ass:greedy}
			For $k=1,\dots,n$, define the metric $d_k$ on $\bb{R}^{m_k}$ through \eqref{eq:metric}
			and suppose that $d_1,\dots, d_n$ are such that for all $k$, $t\in[t_{k-1},t_k)$ $x\in\bb{S}\setminus L_k^{-1}v_k$, there is a reaction $\ell\in\scr{R}$ such that $\lambda_\ell(t,x)>0$ and $d_k(v_k,L_k(x+\xi_\ell))<d_k(v_k, L_k x)$.
		\end{ass}
	
	Intuitively, \Cref{ass:greedy} is satisfied when there is always a reaction available that takes the process closer to the first desired conditioning given being in state $x$ at time $t$. The choice for $g$ will then guarantee that eventually, the path will jump to this point closer to the conditioned state and will therefore hit the observation in finitely many reactions. This argument is formalized in \Cref{thm:Equivalence}. 

		\begin{thm}
				\label{thm:Equivalence}
				Suppose \Cref{ass:greedy} is satisfied. Then $\bb{P}^g_{t_n}(A_n)=1$. 
		\end{thm}

		\begin{cor}
			It follows from \Cref{prop:sets} and \Cref{thm:Equivalence} that, under \Cref{ass:greedy}, $\bb{P}^{g}_{t_n}(A[X^{t_n}])=1$. 
		\end{cor}	

		The choice for $g$ described in this section has two advantages. The first being that the set $A[X^{t_n}]$ on which absolute continuity is obtained is larger than  $A_n$ and the second being that an additional assumption yields equivalence. However, a representation such as  \eqref{eq:htilde-multiple} is not available in this case. When evaluating \eqref{eq:htilde-multiple}, the matrix products and inversions are computed for matrices of size $d\times d$, while \eqref{eq:htilde-multiple-alternative} requires matrix computations where the size increases with the amount of observations. We thus conclude that from a theoretical point of view we may prefer  the construction in this section, but for computational purposes \eqref{eq:htilde-multiple} should be preferred, especially in case $n$ is large.
			\section{Simulation methods}
		\label{sec:simulations}
In this section we discuss methods for simulating the guided process on $[0,T]$. This is a Markov jump process with time-dependent intensity. In case $\|LX(t)-v\|$ is large for $t$ close to $T$, this intensity may blow up. 
 		 
 		 Simulating a Markov jump process is easy when intensities do not depend on time. In this case, given a state $x$ at time $t$, we simply simulate reaction times $\tau_\ell\sim\mathrm{Exp}(\lambda_\ell(x))$ for each reaction, set $\hat{\ell} = \argmin_{\ell\in\scr{R}} \tau_\ell$ and move $t\gets  t+\tau_{\hat{\ell}}$ and $X(t+\tau_{\hat{\ell}})\gets X(t)+\xi_{\hat\ell}$. When $\scr{R}$ contains many reactions, one could alternatively use Gillespie's algorithm, see e.g.\ \cite{gillespie1976, gillespie1977}, which first samples the  reaction time and subsequently the reaction that takes place at that time. 
				
To extend this method to chemical reaction processes with time-dependent intensities, we have to sample reaction times satisfying \eqref{eq:JumpDistribution-Inhomogeneous}. In general this is hard and therefore therefore we consider a Poisson thinning step. This gives \Cref{alg:NextReactionThinning}. The efficiency of \Cref{alg:NextReactionThinning}  will depend on whether sharp bounds $\bar\lambda_\ell$ can be derived. 
		
		\begin{algorithm}
			\caption{Next reaction method for time-inhomogeneous rates with a Poisson thinning step}
			\label{alg:NextReactionThinning}
			\KwIn{$x\in\bb{S}$ and $t\geq 0$ and an upperbound $\bar{\lambda}_\ell$ for $\lambda_\ell(\cdot, x)$.}
			\KwResult{The next reaction $\hat\ell\in\scr{R}$ and the corresponding reaction time $\tau_{\hat{\ell}}$ from the state $(t,x)$.}
			\For{$\ell\in\scr{R}$}{
			Set $t^* = t$ \; 
				Sample $\tau_\ell^* \sim \mathrm{Exp}\left(\bar{\lambda}_{\ell}\right)$ and set $t^*\gets t^*+\tau_\ell^*$\; \label{line:findtime-thinning}
				Sample $U\sim\mathrm{Unif}(0,1)$ \;
				\eIf{$U\leq {\lambda_\ell(t^*,x)}/{\bar{\lambda}_\ell}$}{
					Accept, set $\tau_\ell=t^*-t$ \;
				}{
					Reject and return to line \ref{line:findtime-thinning}  \;
				}\label{line:thinning-endfor}
			}
			Set $\hat\ell = \argmin_{\ell\in\scr{R}} \tau_\ell$ \;
			\Return $\hat\ell$ and $\tau_{\hat{\ell}}$.
		\end{algorithm}
		
		\subsection{Simulation of the guided process for underlying processes with time-homogeneous intensities}
		In many applications, such as the examples in \Cref{subsec:introduction-examples}, the intensities $\lambda_\ell,\, \ell\in\scr{R}$ of the underlying process only depend on the state $x$ and don't have a direct dependence on $t$. Here we consider this case and present a method for simulating the guided process in such a scenario. Note that if the underlying rates are time-dependent, but bounded, these upper bounds can easily be included in the thinning step. 
	
		\subsubsection{Special case: guided process induced by $g_\eps$}	 
We consider the guided process with $\lambda_\ell^{g_\eps}=\alpha^{g_\eps}_\ell\lambda_\ell$, where 	$\alpha^{g_\eps}_\ell$ is defined in \Cref{eq:guided_rate_worked_out}. As  the map $t\mapsto \alpha_\ell^{g_\eps}(t,x)$	
		is monotone, given $(t,x)$, an  upper bound can be found at either $(t,x)$ or $(T,x)$. Unfortunately, the upper bound at $T$ is typically far from sharp, hampering efficiency, especially  when $\eps$ is chosen to be small. To resolve this issue, we employ \Cref{alg:NextReactionThinningBridge}. 	In lines \ref{line:findtime-thinningBridge-generatedelta}--\ref{line:findtime-thinningBridge-generatedeltaEnd}, $\lambda_\ell^{g_\eps}(\cdot,x)$ is increasing in time. Thus we can choose $\delta$ such that $t+\delta<T$ and use that, on the interval $[t, t+\delta]$, $\lambda_\ell^{g_\eps}(\cdot, x)$ is upper bounded by $\lambda_\ell^{g_\eps}(t+\delta, x)$.  Hence, on $[t,t+\delta]$, we apply the thinning property for an inhomogeneous Poisson process with rate $\lambda_\ell^{g_\eps}(\cdot, x)$ and, if no reaction occurs in this interval, we move $t$ to $t+\delta$. Typically, once $LX(t)=v$ for $t$ near $T$ the map $s\mapsto g(s,x+\xi_\ell)/g(s,x)$ will be decreasing for choices of $g$ considered in this article.  If at line \ref{line:findtime-thinningBridge-generatedelta} $LX(t) \neq v$ {\it and} $T-(t+\delta)\le \eps$ we reject the sampled path for specified small $\eps>0$. 
	
		\begin{algorithm}[h]
		\caption{Next reaction method with a Poisson thinning step for guided rates where the original intensity is independent of time}
		\label{alg:NextReactionThinningBridge}
		\KwIn{$x\in\bb{S}$ and $t\geq 0$. }
		\KwResult{The next reaction $\hat\ell\in\scr{R}$ and the corresponding reaction time $\tau_{\hat{\ell}}$ from the state $(t,x)$.}
		Set $\scr{R}_+=\{ \ell\in\scr{R}\mid\lambda_\ell(x)>0\}$ as the set of possible reactions to occur\;
		\For{$\ell\in\scr{R}_+$}{
			\eIf{$s\mapsto \frac{g(s,x+\xi_\ell)}{g(s,x)}$ is decreasing in $s$ on $[t,T)$}{
				Apply lines \ref{line:findtime-thinning}-\ref{line:thinning-endfor} of \Cref{alg:NextReactionThinning} with $\bar{\lambda}_\ell = \lambda_\ell^g(t,x)$\;
			}{
				Set $t^* = t$ and $t_{\mathrm{start}} = t$ \;
				Choose $\delta>0$ such that $t+\delta < T$ \; \label{line:findtime-thinningBridge-generatedelta}
				Sample $\tau_\ell^*\sim \mathrm{Exp}\left(\lambda_\ell^{g}(t+\delta, x)\right)$ and set $t^* \gets t^*+\tau_\ell^*$ \; \label{line:findtime-thinningBridge-generatetau}
				Sample $U\sim\mathrm{Unif}(0,1)$ \;
				\eIf{$t^* > t+\delta$}{
					Set $t\gets t+\delta$ and return to line \ref{line:findtime-thinningBridge-generatedelta}
				}{
					\eIf{$U\leq \lambda_\ell^{g}(t^*, x)/\lambda_\ell^{g}(t+\delta,x)$}{
						Accept and set $\tau_\ell = t^* - t_{\mathrm{start}}$\;
					}{
						Reject and return to line \ref{line:findtime-thinningBridge-generatetau}
					}
				}\label{line:findtime-thinningBridge-generatedeltaEnd}
			}
		}
		Set $\hat\ell = \argmin_{\ell\in\scr{R}_+} \tau_\ell$ \;
		\Return $\hat\ell$ and $\tau_{\hat{\ell}}$.
	\end{algorithm}
				
		\subsubsection{Choice of $\delta$} 
We now explain that  $\delta$ can be chosen such that a minimum acceptance rate	is attained. 

The acceptance rate of a proposed time $\tau_\ell$ is given by 
		\begin{equation} 
			\label{eq:choicedelta-eta1}
			\eta = \frac{\lambda_\ell^{g_\eps}(t+\tau_\ell,x)}{\lambda_\ell^{g_\eps}(t+\delta,x)} 
			\end{equation}
		A direct computation yields that 
		\begin{equation}\label{eq:choicedelta-eta2}
			\eta = \exp\left( \frac{d(v,L(x+\xi_\ell))^2-d(v,Lx)^2}{2(T+\eps-t-\delta)} -  \frac{d(v,L(x+\xi_\ell))^2-d(v,Lx)^2}{2(T+\eps-t-\tau_\ell)}\right),
		\end{equation}
		Upon solving \eqref{eq:choicedelta-eta2} for $\delta$, we find that 
		\[ \delta = T+\eps-t-\left(\frac{2\log\eta}{d(v,L(x+\xi_\ell))^2-d(v,Lx)^2} + \frac{1}{T+\eps-t-\tau_\ell}\right)^{-1} \]
		Now $\tau_\ell$ is not known when choosing $\delta$, but since $\tau_\ell \geq 0$, we obtain the  desired acceptance rate $\eta$ by choosing
		\begin{equation}
			\label{eq:choicedelta-delta}
			\delta\geq T+\eps-t-\left(\frac{2\log\eta}{d(v,L(x+\xi_\ell))^2-d(v,Lx)^2} + \frac{1}{T+\eps-t}\right)^{-1} .
		\end{equation}

		\subsection{Simulation studies}
			
			Julia written code for the simulation examples of this section is  available in  \cite{crmp}.			\subsubsection{Death model and comparison to \cite{golightly2019efficient}}
				We consider the pure death process of \Cref{ex:deathprocess} and estimate the  distribution of $X(T)$ conditional on $X(0)=x_0$. For convenience, we denote by $g^v$ the guiding term $g$ chosen for the conditioning $X(T)=v$. By \Cref{corr:likelihood}, for any $v\in\bb{S}$
				\[ p(0,x_0;T,v) = g^v(0,x_0)\bb{E}^{g^v} \left[\frac{\Psi_T^{g^v}(X)}{g^v(T,X(T))}\ind_{\{X(T)=v\}} \right] \]
				Hence, for large $N$, we can estimate $p(0,x_0;T,v)$ by sampling $X_1,\dots, X_N$ from $\bb{P}^{g^v}$ using \Cref{alg:NextReactionThinningBridge} and computing 
				\begin{equation}
				\label{eq:deathprocess-estimators}
				\hat{p}(v) = \frac{1}{N} \sum_{i=1}^N g^v(0,x_0)\frac{\Psi_T^{g^v}(X_i)}{g^v(T, X_i(T))}\ind_{\{X_i(T)=v\}} 
				\end{equation}

It is a well-known result that $X(T)  \sim \mathrm{Binom}\left(x_0, e^{-cT}\right)$, so we will use this mass function for {comparison. In} the following experiment, we use \eqref{eq:deathprocess-estimators} to estimate the {mass function} of the $\mathrm{Binom}\left(x_0, e^{-cT}\right)$-distribution with $x_0=50$, $T=1$ and $c=1/2$. We consider four choices for $g$: 
\begin{itemize}
	\item $g_F$ as in \eqref{eq:g_fearnhead} with $C= 10^{-5}$ for $v<32$ and $C=0.3$ for $v\geq32$ (we comment on this choice of $C$ below).  
	\item $g_{\mathrm{LNA}}$ from the LNA method with restart as in \eqref{eq:guiding_term_LNA}, with $C =  10^{-5}$.
	\item  $g_\eps$ chosen using a scaled diffusion as in  \eqref{lem:g_epsilon}. We took $\eps = 10^{-5}$. For a given value of $v$, the tuning parameter $a=\sigma^2$ appearing in $g_\eps$  was chosen as $2.5(x_0 -v)$. This was found to be  roughly the optimum of the map $a\mapsto g^v(0,x_0)\bb{E}^{g_{\eps}}\Psi_T^{g^v}(X)\ind_{\{X(T)=v\}}/g^v(T,X(T))$, where the expectation was estimated using $100$ forward simulated paths and  $v$ was taken to be the $1\%$, $50\%$ and $99\%$-quantiles  of $X(T)$. 

	\item $g$ chosen using the density of a reversed (decreasing) constant rate {Poisson} process with rate constant $\theta$. We took $\theta=c v$ (this  choice ensures equivalence by  Theorem 4.2 of \cite{corstanje2021conditioning}). 
\end{itemize}
	For forward simulation under $\bb{P}^g$ we employed \Cref{alg:NextReactionThinningBridge} where for $g_\eps$, $\delta$ was chosen according to (\ref{eq:choicedelta-delta}). For the other choices of $g$  such a closed form expression cannot be derived, and  for  $t<T$ we simply took  $\delta = \frac{T-t}{2}$. This does not affect the validity of the algorithm, only the acceptance probability of a sampled reaction time. 
				
{\it Results:}		We took Monte-Carlo sample size $N=15000$. 		The estimated probability mass functions are depicted in  \Cref{fig:deathprocess-density_estimates}. This figure confirms that $\mathbb{P}_T^h\ll \mathbb{P}_T^g$ for all choices of $g$ considered. 
The percentages of paths conditioned on $X(T)=v$ and actually ending in $v$ are depicted in \Cref{fig:deathprocess-hitting_probabilities}
  for a range of values of $v$. We now comment on the choice of $C$ for the Fearnhead guiding term. For $C=10^{-5}$ we noticed that conditioning on high values of $v$ caused the guided process to  jump very closely to $T$ with high probability. This resulted in numerical instability when computing the integral appearing in $\Psi_T^g$. For that reason, we took a larger value of $C$ in case $v$ is large. A side effect of that is that there is less strong forcing to hit $v$ at time $T$ which explains the lower fractions of paths ending in $v$ for $v\ge 32$. 

We observe that the overall error of the Poisson guiding term is lowest and the diffusion guiding term performs particularly well for low values of $v$ for which the process has to make a lot of jumps. However, when conditioning on high values, the amount of sample paths that end up in the correct state is low, leading to a larger variance of $\hat{p}(v)$ defined in  \Cref{eq:deathprocess-estimators}.  The performance of the Poisson guiding term is explained due to the typically lower guiding term, which is exactly $0$ when the conditioning is reached at a time prior to $T$. This property is not shared by the other choices considered.

In \Cref{tab:deathprocess_squared_errors} we report the mean squared errors $\frac{1}{\#v}\sum_v (q(v)-\hat{p}(v))^2$, with $q$ denoting the probability mass function of the $\mathrm{Binom}(x_0,e^{-cT})$-distribution. 
				
				\begin{figure}[h]
				\centering
				\includegraphics[width = \textwidth]{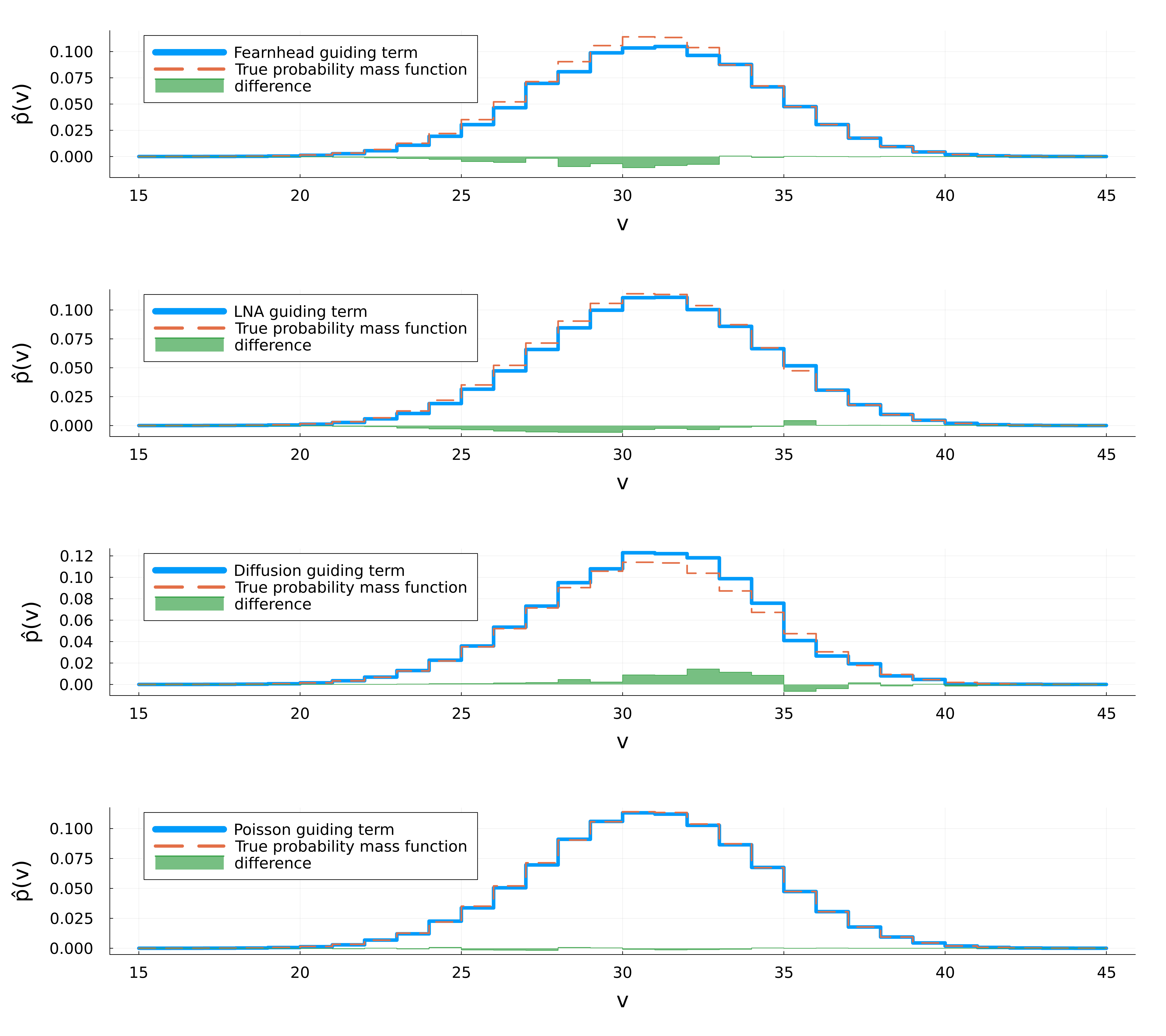}
				\caption{Estimates of the probability mass function of $X(T)\mid X(0)=x_0$ using the guiding functions $f$ from the LNA method, a scaled diffusion and a Poisson process. The upper barplot is the true density. For each $v$, we estimated using \eqref{eq:deathprocess-estimators} with $N = 15000$.}\label{fig:deathprocess-density_estimates}
				\end{figure}
				
				\begin{figure}[h]
				\centering
				\includegraphics[scale=0.8]{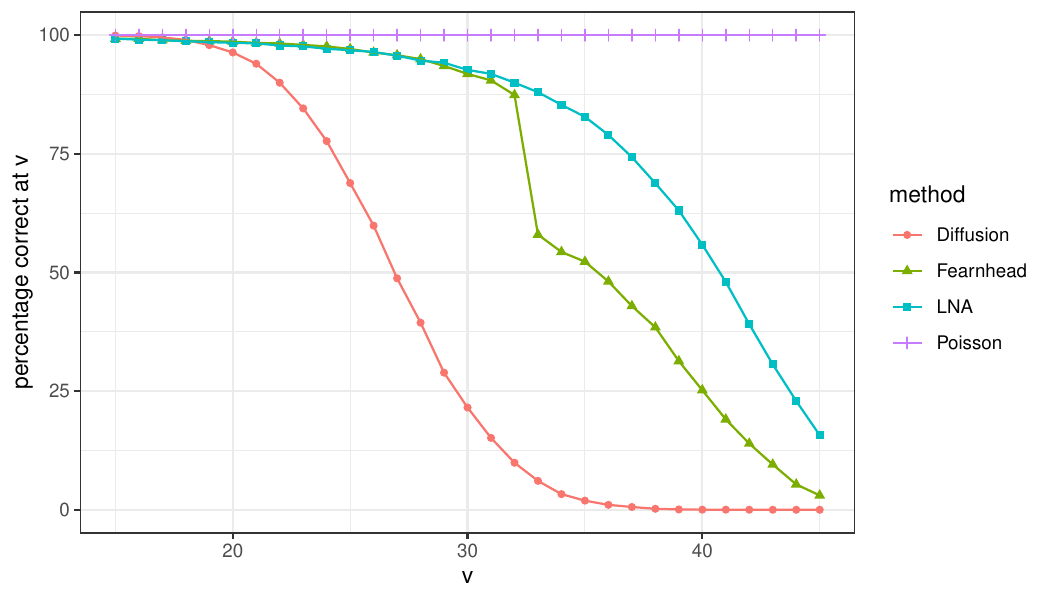}
				\caption{The percentage of paths ending in the point of conditioning ($v$) versus $v$ for the four  methods considered  (with the same settings as in \Cref{fig:deathprocess-density_estimates}).}
				\label{fig:deathprocess-hitting_probabilities}
				\end{figure}

				\begin{table}[h]
				\caption{Mean squared errors of the estimates from \Cref{fig:deathprocess-density_estimates} for each of the methods}
				\begin{tabular}{|c|c|c|}
					\hline
					Method & Mean Squared Error \\ \hline 
					Fearnhead & $1.5\cdot 10^{-5}$ \\
					LNA (with restart) & $6.5\cdot 10^{-6}$ \\
					Diffusion guiding term & $2.1\cdot 10^{-5}$ \\
					Poisson guiding term & $4.0\cdot 10^{-7}$ \\  \hline
				\end{tabular}
				\label{tab:deathprocess_squared_errors} 
				\end{table}

		\subsubsection{Gene transcription and translation}
		
			We consider the \textit{GTT}-model presented in \Cref{ex:GTT}. \Cref{fig:GTTBridge-oneObservation} contains realizations of the process $X$ under $\bb{P}^{g}$ conditioned to hit $(1,11, 56)$ at time $T=1$. We used guiding induced by $g_\eps$ with $\eps=10^{-5}$ and $a = a_{\mathrm{CLE}}(0,x_0)$. The plot shows 10 sampled trajectories. Out of $1000$ more trajectories sampled, all of these end in the correct point $X(T)=(1,11,56)$.

		\begin{figure}[h]
			\centering
			\includegraphics[width = \textwidth]{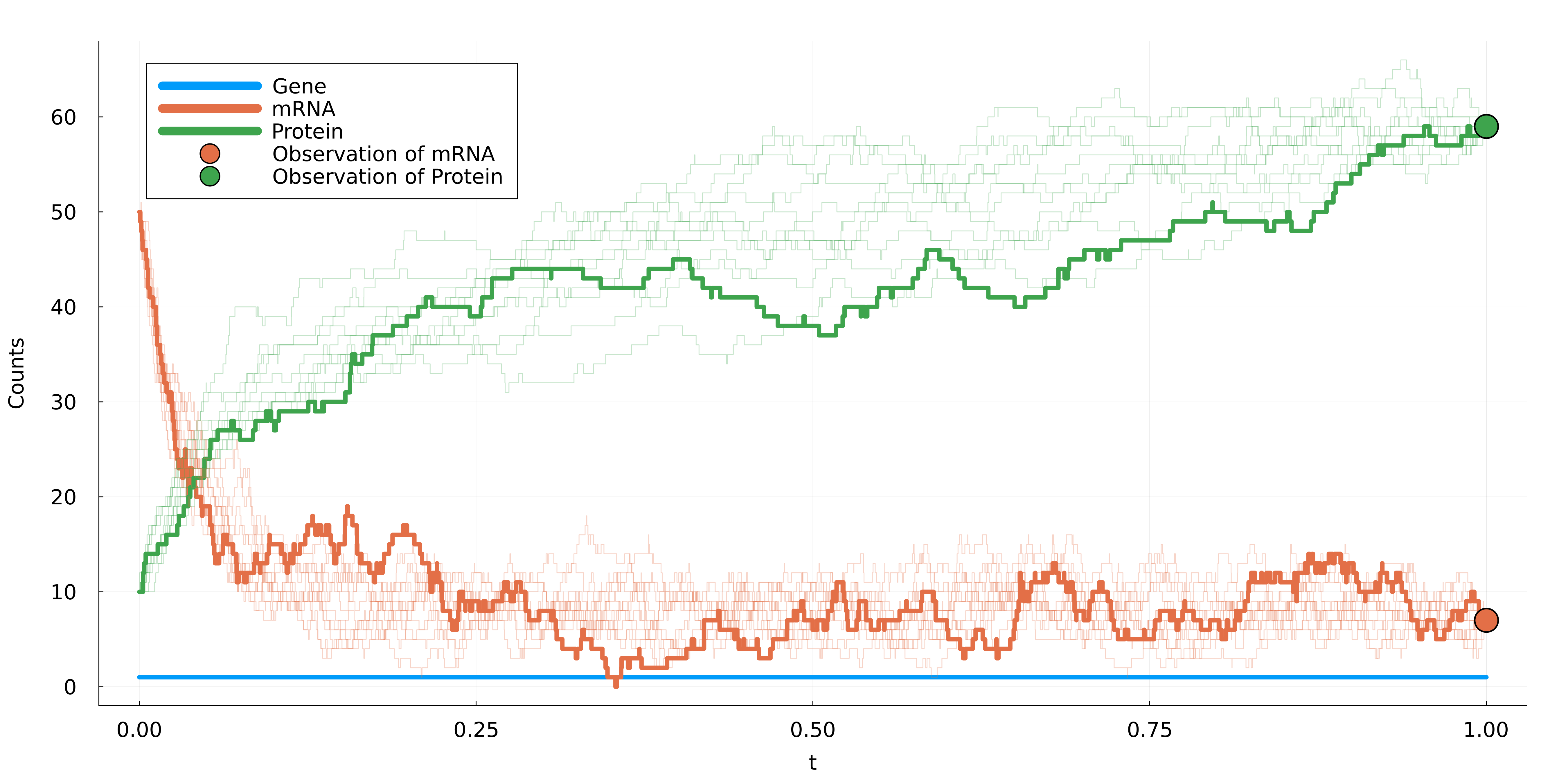}
			\caption{Realization of a guided process starting from $x_0=(1,50,10)$ conditioned to be at $(1,10,50)$ at time $T=1$ of the \textit{GTT}-model from \Cref{ex:GTT} with $\kappa_1 = 100$, $\kappa_2 = 10$, $d_M = 25$ and $d_P=1$. The thick line is the original  (unconditioned) process.}
			\label{fig:GTTBridge-oneObservation}
		\end{figure}

	Next, from a simulated forward path 	
		we chose  15 times at random and saved the values of randomly chosen components of the process. Taking these values as observations  we show in \Cref{fig:GTTBridge-15observations} multiple realisations of the guided process. Out of  $1000$  simulated trajectories we found that all of those passed through  all partial observations.
		
		\begin{figure}[h]
			\centering
			\includegraphics[width = \textwidth]{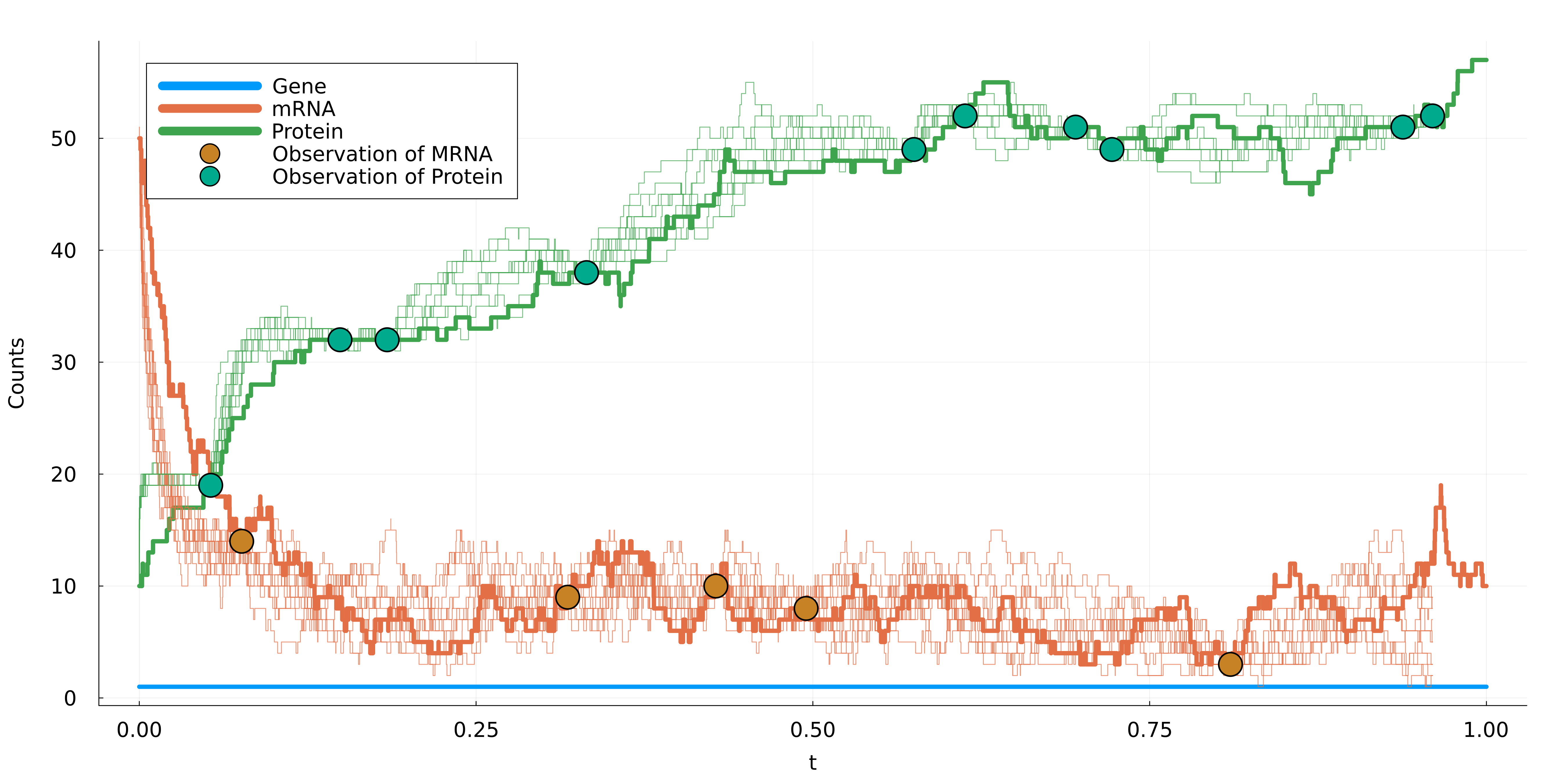}
			\caption{Realization of a guided process through 15 arbitrary observations, with the same settings as in \Cref{fig:GTTBridge-oneObservation}.}
			\label{fig:GTTBridge-15observations}
		\end{figure}

	\subsubsection{Enzyme kinetics}
		As described in \Cref{ex:Enzyme-kinetics}, this is a very interesting example as it contains a monotone component as well as absorbing states, where no more reactions can occur. For example, the state $X(t) = (0,20,0,32)$ can be reached with a sequence of reactions from $x_0$, but no further reactions are possible in this state. 
		We use this example to compare various guiding terms. 
		We choose the starting point $x_0 = (12,10,10,10)$ and parameters $(\kappa_1, \kappa_2, \kappa_3) = (5,5,3)$. We consider three scenarios in which the LNA method with restart, the guiding term obtained from a scaled diffusion presented in \Cref{subsec:guiding_term_scaled_diffusion} and the guiding term in which one of the components is replaced by a Poisson guiding term as presented in \Cref{sec:poisson}. 
		
		\begin{itemize}
		\item \textbf{Scenario A: }	The process is conditioned to be at $x_{T,0.01} = (0,15,5,27)$ at time $T=1$. $27$ is the $1\%$ quantile of $X_4(T)$, determined through forward simulation. 
		\item \textbf{Scenario B: } The process is conditioned to be at $x_{T,0.5} = (0,19,1,31)$ at time $T=1$. $31$ is the $50\%$ quantile of $X_4(T)$, determined through forward simulation.
		\item \textbf{Scenario C: } The process is conditioned to be at $x_{T,0.99} = (0,20,0,32)$ at time $T=1$. $32$ is the $99\%$ quantile of $X_4(T)$, determined through forward simulation and $x_{T,0.99}$ is also an absorbing state of the process. 
		\end{itemize}
		
		First we simulate 100 trajectories of each process and check the percentage of paths that satisfy $X(T) = x_{T,q}$ for $q=0.01,0.5,0.99$. We used the same scheme for simulating the guided process as for the death process. For the diffusion guiding term, we took $a = a_{\mathrm{CLE}}(0,x_0)$. For the Poisson guiding term, the auxiliary process from which $g$ is obtained contains a scaled diffusion in the $3$ components and a Poisson processes in the third component. The diffusion is scaled by a matrix $a$, in which we used the first three rows and the first three columns of $a_{\mathrm{CLE}}(0,x_0)$. The intensity of the Poisson component was taken as $\lambda_3(0,x_0)$, which is a lower bound of the reaction intensity for the reaction that induces the monotone component.  The results are summarised in \Cref{fig:percentages}.
		
		\begin{figure}[h]
		\centering
		\includegraphics[width = 0.8\textwidth]{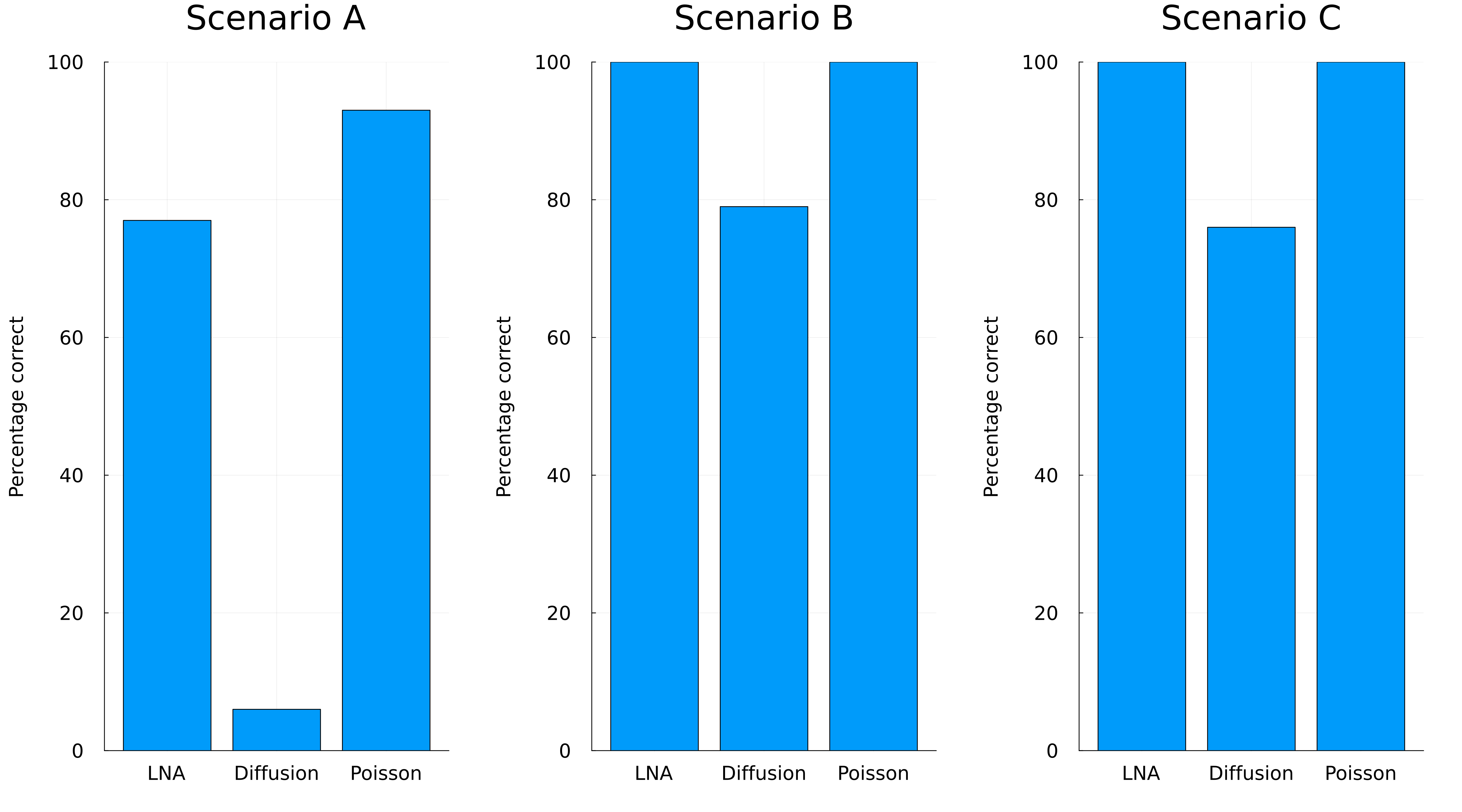}
		\caption{Percentages out of $100$ trajectories that satisfy the conditionings set in scenario's \textbf{A}, \textbf{B} and \textbf{C}.}
		\label{fig:percentages}	
		\end{figure}

Similar to the death process in \Cref{fig:deathprocess-density_estimates}, we see that the diffusion guiding term and the LNA guiding term struggle when the monotone component is conditioned  not to make many jumps.
		
We used the same method as for the death processes to estimate $p(0,x_0 ; T, v)$  for scenario's A, B and C. In each scenario, we computed $\hat{p}$ through \eqref{eq:deathprocess-estimators}. 
		In each scenario, we make 200 estimates for $p(0,x_0 ; T, v)$  by computing $\hat{p}$ with $N=1000$ processes and we compared the LNA method (without restart cf. Section 4.3 of \cite{golightly2019efficient}) with the diffusion guiding term with $\eps = 10^{-5}$ and $a = 100a_{\mathrm{CLE}}(T,x_{T,q})$ for $q=0.01,0.5,0.99$. For the LNA method we assumed extrinsic noise $C=2000I$, $C=500I$ and $C=200I$ for scenario's A, B and C, respectively. Histograms of the estimates are given in \Cref{fig:enzyme-kinetics-estimates-A}, \Cref{fig:enzyme-kinetics-estimates-B} and \Cref{fig:enzyme-kinetics-estimates-C}. The MSEs of the estimates are summarised in \Cref{tab:enzyme-kinetics-MSEs}. 
		
		\begin{figure}[h]
		\centering
		\includegraphics[width = 0.6\textwidth]{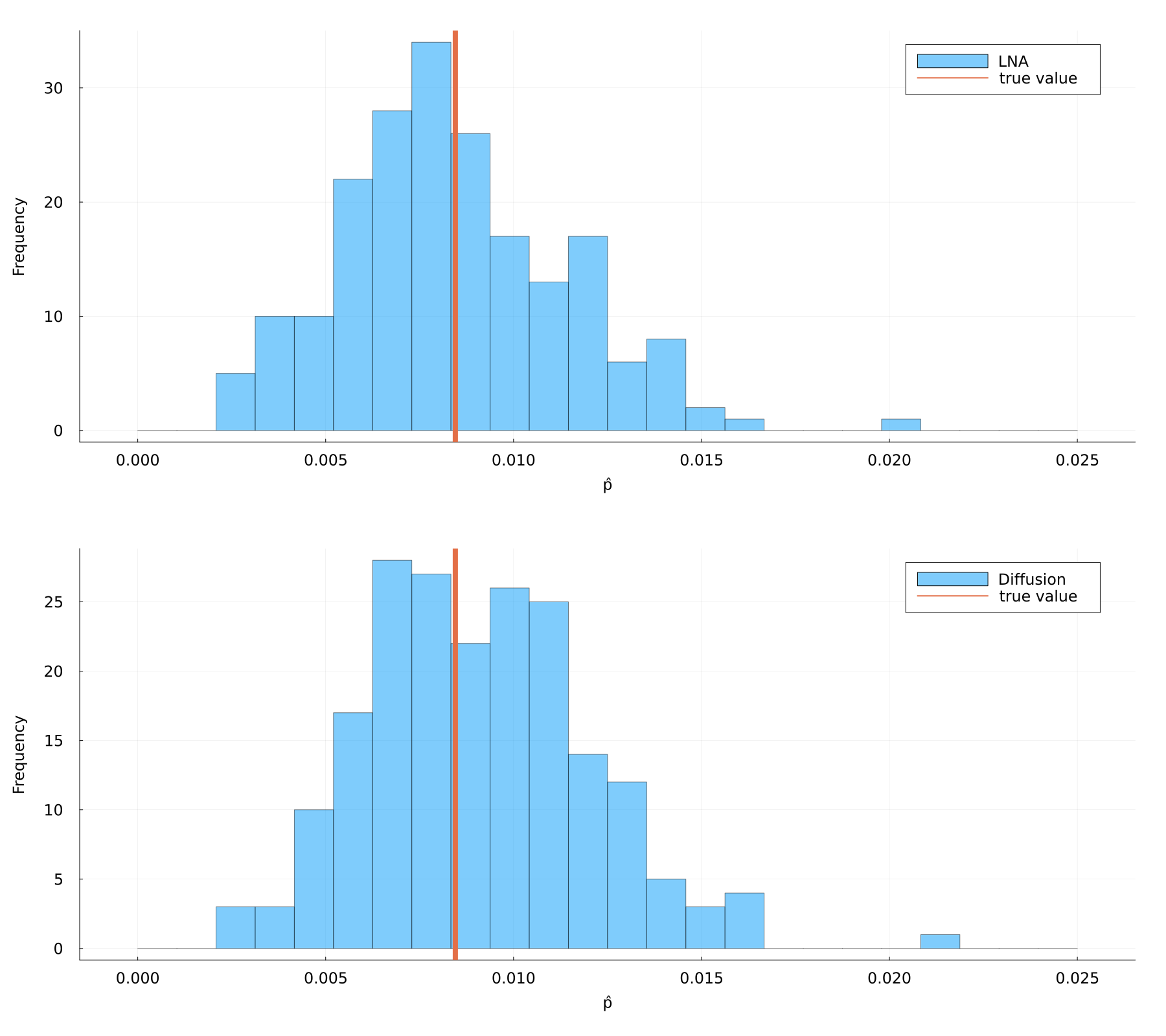}
		\caption{Histograms of $200$ estimates for $p(0,x_0;T,v)$ in scenario A using the LNA method without restart (top) and and the diffusion method (bottom). The true value was estimated using 10,000 forward simulations. }
		\label{fig:enzyme-kinetics-estimates-A}
		\end{figure}
		
		\begin{figure}[h]
		\centering
		\includegraphics[width = 0.6\textwidth]{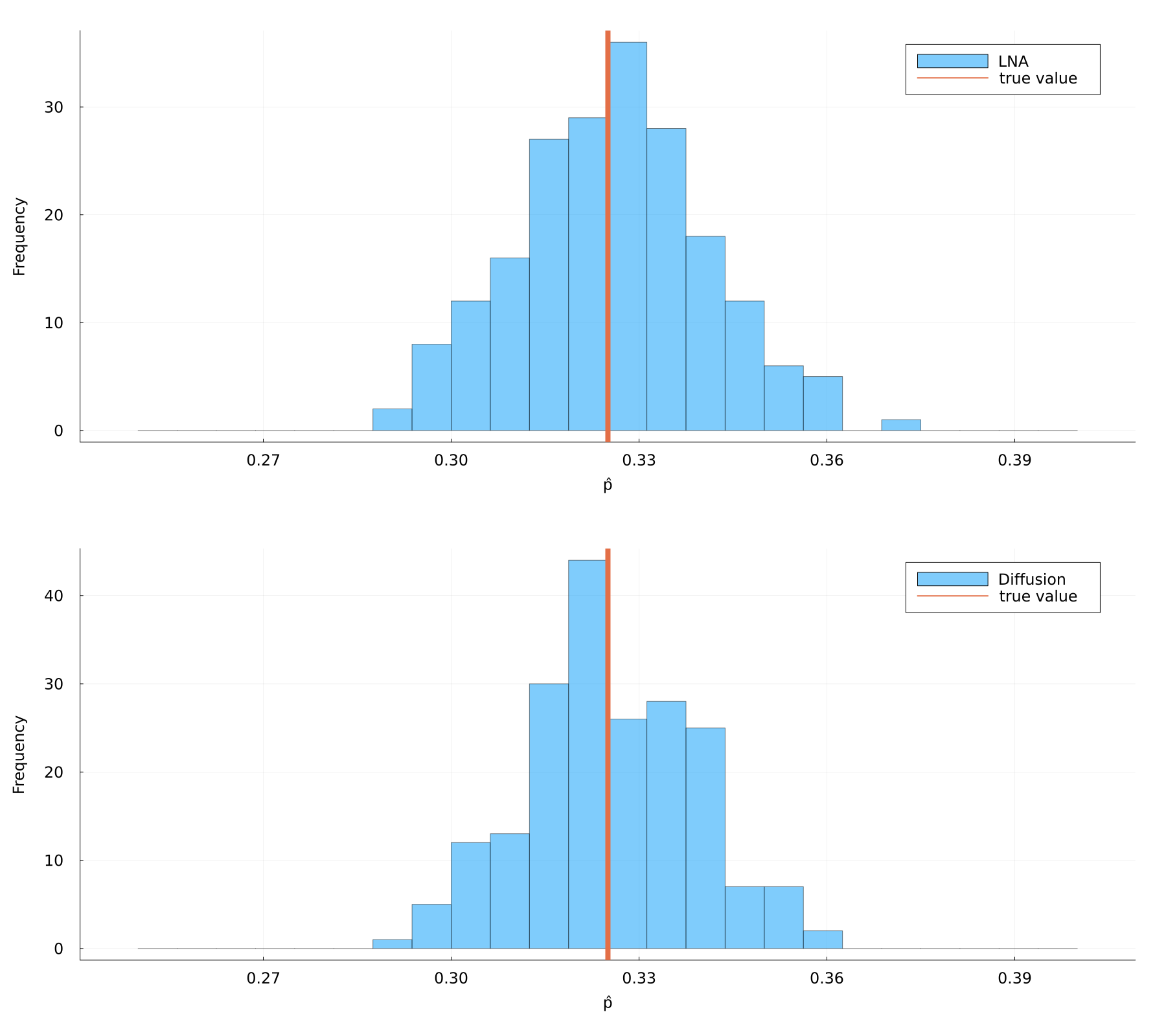}
		\caption{Histograms of $200$ estimates for $p(0,x_0;T,v)$ in scenario B using the LNA method without restart (top) and and the diffusion method (bottom). The true value was estimated using 10,000 forward simulations. }
		\label{fig:enzyme-kinetics-estimates-B}
		\end{figure}		
		\begin{figure}[h]
		\centering
		\includegraphics[width = 0.6\textwidth]{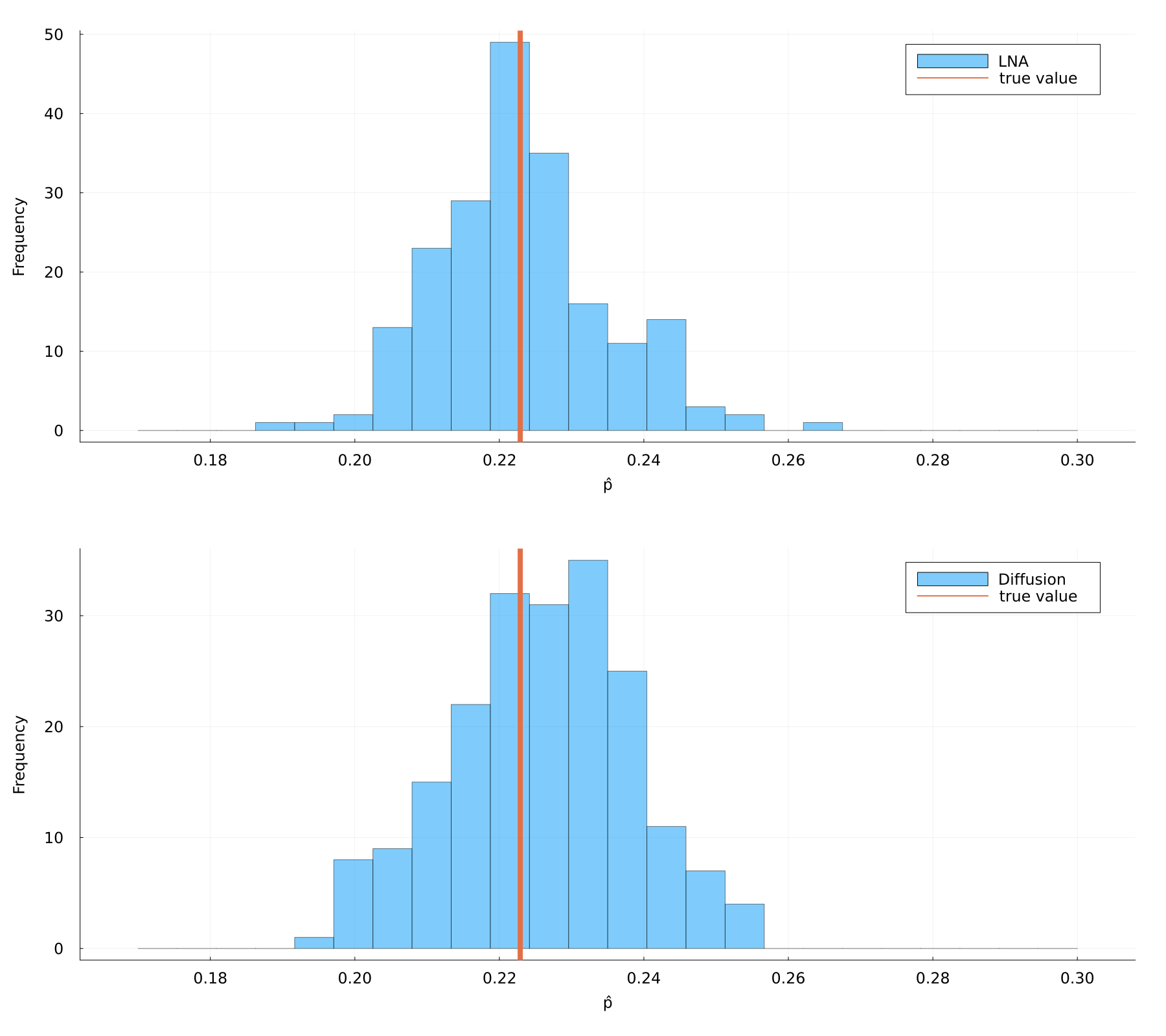}
		\caption{Histograms of $200$ estimates for $p(0,x_0;T,v)$ in scenario C using the LNA method without restart (top) and and the diffusion method (bottom). The true value was estimated using 10,000 forward simulations. }
		\label{fig:enzyme-kinetics-estimates-C}
		\end{figure}
		
			\begin{table}[h]
						\caption{Mean  Squared 
						Error (MSE) of the estimates from Figures \ref{fig:enzyme-kinetics-estimates-A}, \ref{fig:enzyme-kinetics-estimates-B} and \ref{fig:enzyme-kinetics-estimates-C} for each of the methods}
				\begin{tabular}{|c|c|c|c|}
					\hline
					Method & MSE scenario A & MSE scenario B & MSE scenario C \\ \hline 
					LNA (without restart) & $8.8\cdot 10^{-6}$ & $2.3\cdot 10^{-3}$ & $1.3 \cdot 10^{-3}$\\
					Diffusion guiding term & $9.0\cdot 10^{-6}$ & $1.8 \cdot 10^{-3}$ & $1.6 \cdot 10^{-3}$ \\ \hline
				\end{tabular}
	
				\label{tab:enzyme-kinetics-MSEs} 
				\end{table}
\section{Discussion}

In this paper we have provided sufficient conditions for constructing valid guided processes, where ``valid'' refers to 
{the law of the true conditioned process being absolutely continuous with respect to the law of the guided process induced by $g$ with the laws as defined in \Cref{sec:doobh}}. We have presented various choices of $g$ and conclude that among those there is no best choice in terms of closeness to the true conditioned process and computational cost. When used within a sequential Monte Carlo or Markov chain Monte Carlo algorithm, a mixture of proposals may be beneficial. The inherent discreteness of chemical reaction processes makes it a hard problem, but it works to our advantage in the sense that guided processes can be constructed on $[0,T]$ where $g(T,x)$ is well defined. This is accomplished by using a guiding term derived from conditioning a diffusion process that is observed with (small) extrinsic noise. The guiding term defined in \Cref{sec:poisson}  has the additional advantage that   monotonicity can be exploited for   simulating paths of the guided process efficiently. 

In the analogous problem of continuous-discrete smoothing for diffusion processes (see e.g.\ \cite{mider2020continuousdiscrete}, \cite{beskos-mcmc-methods}, \cite{golwilk}) there exists a simple random-walk like MCMC-sampler on path space to update guided processes: the preconditioned Crank-Nicolson scheme. Unfortunately, we are not aware of a similar construct for chemical reaction processes.

	\section{Acknowledgments}
	The authors would like to thank the anonymous referees for their helpful comments. We also thank M.\ Schauer (Gothenburg University and Chalmers University of Technology, Sweden) for stimulating discussions on the simulation algorithms. 
	
	\section{Declarations}
		This work is part of the research project "Bayes for longitudinal data on manifolds" with project number OCENW.KLEIN.218, which is financed by the Dutch Research Council (NWO).
				The authors have no competing interests to declare that are relevant to the content of this article.
	
	\appendix

	\section{Proof of \Cref{thm:positive-prop-1obs}}
	\label{app:proof-of-path-exists}

	The proof relies on the existence of a chain of reactions that result in the conditioning $LX(T)=v$. We formalize this in \Cref{def:connection}.
	
	\begin{defn}
		\label{def:connection}
		For $x,y\in\bb{S}$, we say that $y$ can be reached from $x$ within a time interval $(a,b)\subseteq(0,T)$, denoted $ x\hookrightarrow y$ in $(a,b)$, if there is a finite collection $\{ \ell_1,\dots,\ell_I\}\subseteq\scr{R}$ of reactions and a partition $U_1,\dots, U_I$ of disjoint intervals for $(a,b)$ such that $y=x +\sum_{i=1}^I \xi_{\ell_i}$ and for all the set $i$, $\mathrm{supp}\, \lambda_{\ell_i}\left(x+\sum_{j<i}\xi_{\ell_j}\right) \cap U_i$ is of positive Lebesgue measure. 
	\end{defn}
	
	Since we consider a process conditioned on $LX(T)=v$, we can assume there exists an $x_T\in\bb{S}$ such that $Lx_T=v$ and $x_0\hookrightarrow x_T$ in $(0,T)$. \\
	
	\begin{proof}[Proof of \Cref{thm:positive-prop-1obs}]	
	 Notice that, given any $t<T$, 
	\[ \bb{P}^{g}\left(X(s)=x_T\text{ for all }s\in[t,T]\mid X(t)=x_T\right) = \exp\left(-\int_t^T\sum_{\ell\in\scr{R}} \lambda_{\ell}^{g}(s,x_T)\dd s \right). \]
	
	Now it can be shown that $\sum_{\ell\in\scr{R}} \lambda_\ell^{g}$ is bounded through \Cref{lem:Bounds-on-htilde} and Assumption \Cref{ass:AssumptionsOnProcess-FiniteSum} that the right hand side is strictly positive and hence, it suffices to show that $\bb{P}^{g}\left(\exists t<T : X(t)=x_T\right)>0$. In the remainder of this proof, we show this by showing that the probability that the chain of reactions described in \Cref{def:connection} connection $x_0$ to $x_T$ occurs is positive. \\
	
	Let $\ell_1,\dots,\ell_I$ be the collection of reactions described in \Cref{def:connection} and let $U_1,\dots,U_I$ be the respective partition of $(0,T)$. For the first reaction to occur, we denote by $T_\ell^{(1)}$ the reaction time for each reaction $\ell\in\scr{R}$ and let $L_1$ be the first reaction to occur. That is,
	\[ L_1 = \argmin_{\ell\in\scr{R}} T_\ell^{(1)},\qquad\text{where}\qquad \bb{P}^{g}\left(T_\ell^{(1)}>t\right) = \exp\left(-\int_0^t \lambda_\ell^{g}(s,x_0)\dd s\right). \]
	Note that,
	\[ \left\{ T_{\ell_1}^{(1)} \in U_1,\, T_\ell^{(1)}>\sup\, U_1 \text{ for all }\ell\in\scr{R}\setminus\{\ell_1\}\right\} \subseteq \left\{ T_{\ell_1}^{(1)}\in U_1,\, L_1=\ell_1 \right\} \]
	and thus 
	\begin{equation*}
	\begin{gathered}
		\bb{P}^{g}\left( T_{\ell_1}^{(1)} \in U_1,\, L_1=\ell_1\right) \geq \\
		\left( 1- \exp\left(-\int_{U_1}\lambda^{g}_{\ell_1}(s,x_0)\dd s \right)\right)\exp\left( -\int_{U_1} \sum_{\ell\in\scr{R}\setminus\{\ell_1\}} \lambda_\ell^{g}(s,x_0)\dd s\right).
		\end{gathered}
	\end{equation*}
	Since $\lambda_{\ell_1}^{g}(\cdot,x_0)$ has support in $U_1$, the first term is nonzero, while the second term is nonzero by Assumption \ref{ass:AssumptionsOnProcess-FiniteSum}. \\
	
	For the second reaction, we set
	\[ L_2 = \argmin_{\ell\in\scr{R}} T_\ell^{(2)} \qquad\text{where}\qquad \bb{P}^{g} \left( T_\ell^{(2)}>t \mid L_1 \right) = \exp\left( -\int_{T_{L_1}^{(1)}}^t \lambda_{\ell}^{g}(s, x_0+\xi_{L_1}) \dd s\right).  \]
	Through similar reasoning, we deduce that 
	\begin{equation*}
		\begin{aligned}
			&\bb{P}^{g} \left( L_2 = \ell_2 ,\, T_{\ell_2}^{(2)} \in U_2,\, L_1 = \ell_1 ,\, T_{\ell_1}^{(1)} \in U_1\right) \\
			& \quad = \bb{P}^{g}\left( L_2 = \ell_2,\, T_{\ell_2}^{(2)} \in U_2 \mid L_1 = \ell_1,\, T_{\ell_1}^{(1)} \in U_1\right) \bb{P}^{g}\left(L_1 = \ell_1,\, T_{\ell_1}^{(1)} \in U_1)\right) \\
			& \quad \geq \left( 1- \exp\left( -\int_{U_2}\lambda_{\ell_2}^{g}\left(s, x_0 + \xi_{\ell_1}\right)\dd s \right)\right)\exp\left( -\int_{U_2} \sum_{\ell\in\scr{R}\setminus\{\ell_2\}} \lambda_{\ell}^{g} (s, x_0+\xi_{\ell_1}) \dd s  \right) \\
			& \quad \quad \times \left( 1- \exp\left(-\int_{U_1}\lambda^{g}_{\ell_1}(s,x_0)\dd s \right)\right)\exp\left( -\int_{U_1} \sum_{\ell\in\scr{R}\setminus\{\ell_1\}} \lambda_\ell^{g}(s,x_0)\dd s\right).
		\end{aligned}
	\end{equation*}
	
	Upon iteratively repeating this process, we find 
	\begin{gather*}
		\bb{P}^{g}\left(L_i = \ell_i,\, T_{\ell_i}^{(i)} \in U_i ,\, i=1,\dots,I\right) \geq \\
		\prod_{i=1}^I \left( 1- \exp\left(-\int_{U_i}\lambda_{\ell_i}^{g}\left(s, x_0+\sum_{j<i}\xi_{\ell_j}\right)\dd s\right) \right) \\
		\times \exp\left(-\int_{U_i}\sum_{\ell\in\scr{R}\setminus\{\ell_i\}}\lambda_{\ell_i}^{g}\left(s,x+\sum_{j<i}\xi_{\ell_j}\right)\dd s\right).
	\end{gather*}
	The right hand size is strictly positive as a finite product of nonzero terms. Therefore
	$ \bb{P}^{g}\left( \exists t\leq T : X(t)=x_T\right)>0 $, which finishes the proof. 
	\end{proof}	
	
	\section{Proofs for \Cref{sec:equivalence}}
	\label{app:proof-of-main}
	
		\begin{thm}[Theorem 3.3 of \cite{corstanje2021conditioning}]
		\label{thm:AbsoluteContinuityTheorem}
		Suppose that there exist a family of $\scr{F}_t$-measurable events $\{A_j[X^t] \}_j$ for each $t\in [0,t_n)$ so that the following assumptions hold. 
		\begin{enumerate}[label={\color{red} (\ref{thm:AbsoluteContinuityTheorem}\alph*)}]
			\item \label{ass:EquivalenceAssumption-EventSequence1}
			For all $j$ $ s\leq t$, $A_j[X^t]\subseteq A_{j+1}[X^t]$ and $A_j[X^t]\subseteq A_j[X^s]$. 
			\item 	\label{ass:LimitForF}
			For all $k$ and $s\in[0,t_k)$ and $x\in\bb{S}$,
			\[ \lim_{t\uparrow t_k} \bb{E}\left[g(t,X(t))\mid X(s)=x\right]  = \bb{E}\left[g(t_k, X(t_k))\ind\{L_kX(t_k)=v_k\}\right] , \]
			Moreover, $\bb{E}\left[g(t,X(t))\mid X(s)\right]$ is $\bb{P}$-almost surely bounded for all $t$.
			\item \label{ass:EquivalenceAssumption2Limitfraction}
			For all $k=1,\dots,n$,
			\[ \lim_{j\to\infty} \lim_{t\uparrow t_k} \bb{E}^{h}_t \left[\frac{g(t, X(t))}{h(t, X(t))} \ind_{A_j[X^t]} \right]= \frac{\bb{E}\left[g(t_k, X(t_k))\ind\{L_kX(t_k)=v_k\}\ind_{A[X^{t_k}]}\right]}{h(0,x_0)},\]
			where $A_j[X^{t_k}] = \cap_{t\in[0,t_k)} A_j[X^t]$, $A[X^t] = \bigcup_j A_j[X^t]$ and $A[X^{t_k}]=\cap_{t\in[0,t_k)} A[X^t]$.
			\item \label{ass:AlmostSureBoundOnPsi}
			For all fixed $j$, $\Psi_t^g(X)\ind_{A_j[X^t]}$ is $\bb{P}^{g}$-almost surely uniformly bounded in $t$. 
		\end{enumerate}
		
		Then for any bounded measurable function $f$, 
		\begin{equation}
			\label{eq:ConditioningByGuiding-earlierObservations}
			\bb{E}\left[ f(X)\frac{ \bb{E}\left[g(t_k,X(t_k))\ind\{L_kX(t_k)=v_k\}\right] }{h(0,x_0)}\ind_{A[X^{t_k}]} \right] = \bb{E}^{g}\left[ f(X) \frac{g(0,x_0)}{h(0,x_0)}\Psi_{t_k}^g(X)\ind_{A[X^{t_k}]} \right]
		\end{equation}
		for $k=1,\dots,n$. In particular, for $k=n$,
		\begin{equation} 
			\label{eq:ConditioningByGuiding-finalObservation}
			\bb{E}^{h} \left[ f(X)\ind_{A[X^{t_n}]} \right] =  \bb{E}^{g}\left[ f(X) \frac{g(0,x_0)}{h(0,x_0)}\Psi_{t_n}^g(X)\ind_{A[X^{t_n}]} \right].
		\end{equation}
	\end{thm}
	\begin{proof}
		The proof of Theorem 3.3 of \cite{corstanje2021conditioning} can be followed with the new limits in \ref{ass:LimitForF} and \ref{ass:EquivalenceAssumption2Limitfraction} for each of the limits $t\uparrow t_k$. 
	\end{proof}

	\begin{proof}[Proof of \Cref{lem:htilde-property}]
	The proof utilizes \Cref{thm:AbsoluteContinuityTheorem} with
		\begin{equation}
			\label{eq:defV-and-A}
			V(t,x) = \sum_{\ell\in\scr{R}} \lambda_\ell^{g}(t,x) \qquad\text{and}\qquad A_j[X^t] = \left\{ \sup_{0\leq s<t} V(s,X(s))\leq j \right\}.
		\end{equation}
	Then \ref{ass:EquivalenceAssumption-EventSequence1} is satisfied by construction. Lemmas \ref{lem:ProofOfEquivalence-lem1}-\ref{lem:ProofOfEquivalence-lem3} prove the remaining conditions. 	
	\end{proof}
	
	\begin{lem}
	\label{lem:limitassumption-htilde} 
	The map $g$, defined in \eqref{eq:htilde-equivalence}, is such that $g\in\scr{G}$ and for all $k=1,\dots,n$ and $x\in\bb{S}$, 
	\begin{equation} 
	\label{eq:limitassumption-htilde}
	\lim_{t\uparrow t_k} g(t,x) = \begin{cases}
 		g(t_k,x) & \text{if }L_kx=v_k \\
 		0 & \text{otherwise}	
 \end{cases}. \end{equation}
 and $\partial_{t}\log g(t,x)\leq 0$ for all $(t,x)\in(0,t_n)\times\bb{S}$.
\end{lem}
\begin{proof}
	Through a direct computation using the Schur complement, it can be derived that as $t\uparrow t_k$,
	\begin{equation}
	\label{eq:behavior-loghtilde}
		(v(t)-L(t)x)^\T M(t)(v(t)-L(t)x) = \frac{d_k(v_k,L_kx)^2}{(t_k-t)} + o(t_k-t),
	\end{equation}
		where $d_k$ is the metric on $\bb{R}^{m_k}$ defined as
			\begin{equation}
				\label{eq:metric}
				d_k(x,y) = \sqrt{(y-x)^\T \left(L_ka_k L_k^\T\right)^{-1}(y-x)}, \qquad x,y\in\bb{R}^{m_k}. 
			\end{equation}
	which yields \eqref{eq:limitassumption-htilde}. Since $v$ and $L$ are piecewise constant and $\der{M}{t} = M(t) \Bm L_ia_k L_j^\T \Em_{i,j=k}^n M(t)$ is positive semidefinite, $\partial_t\log g(t,x)\leq 0$. 
	\end{proof}

	\begin{lem}
	\label{lem:ProofOfEquivalence-lem1}
		For all $k=1,\dots, n$, $s\in[0,t_n)$ and $x\in\bb{S}$,
		\[ \lim_{t\uparrow t_k}\bb{E}\left[ g(t,X(t)) \mid X(s)=x\right] = \bb{E}\left[ g(t_k, X(t_k))\ind\{ L_kX(t_k)=v_k \}\mid X(s)=x\right].  \]
		
	\end{lem}
	\begin{proof}
		By dominated convergence and \Cref{lem:limitassumption-htilde}
		\[ \begin{aligned} 
			\lim_{t\uparrow t_k} \bb{E}\left[ g(t,X(t))\mid X(s)=x \right] &= \lim_{t\uparrow t_k} \sum_{y\in\bb{S}} g(t,y)p(s,x;t,y)\\
			 &= \sum_{y\in\bb{S}} g(t_k,y)\ind\{ L_ky=v_k \}p(s,x;t_k,y)  \\
			 &= \bb{E}\left[ g(t_k,X(t_k))\ind\{L_kX(t_k)=v_k\}\mid X(s)=x\right]  .
		\end{aligned} \]
	\end{proof}

	\begin{lem}
		\label{lem:ProofOfEquivalence-lem2}
		For all $k=1,\dots,n$,
		\[ \lim_{j\to\infty} \lim_{t\uparrow t_k} \bb{E}^{h} \left[ \frac{g(t,X(t))}{h(t,X(t))}\ind_{A_j[X^t]} \right]= \frac{ \bb{E}\left[ g(t_k,X(t_k))\ind\{ L_kX(t_k)=v_k \}\ind_{A[X^{t_k}]} \right] }{h(0,x_0)}. \]
	\end{lem}

	\begin{proof}
		Upon defining $Z(t) = g(t,X(t))$, it follows from dominated convergence and \Cref{lem:limitassumption-htilde} that for any bounded continuous function $f$,
		\begin{equation*}
			\begin{aligned}
				\lim_{t\uparrow t_k} \bb{E} f(Z(t)) &= \sum_{x\in\bb{S}}\lim_{t\uparrow t_k} f\left(g(t,x)\right)p(0,x_0;t,x) \\
				&= \sum_{x\in\bb{S}} f\left(g(t_k,x)\ind\{ L_k x =v_k \}\right)p(0,x_0;t_k,x) \\
				&= \bb{E}f\left( g(t_k,X(t_k))\ind\{ L_k X(t_k) =v_k \} \right).
			\end{aligned}
		\end{equation*}	
		Hence, 
		\begin{equation}
		\label{eq:weakconvergence}	
		Z(t)\dto g(t_k,X(t_k))\ind\{ L_k X(t_k) =v_k \}\text{ as }t\uparrow t_k.
		\end{equation}
 		Now 
		\[ \bb{E}\left[ Z(t)\ind_{A_j[X^t]} \right] = \bb{E}\left[ Z(t)\ind_{A_j[X^{t_k}]}\right] +\bb{E}\left[ Z(t)\ind_{A_j[X^t]\setminus A_j[X^{t_k}]}\right]  \]
		By \eqref{eq:weakconvergence}, the first term tends to $\bb{E}\left[ g(t_k,X(t_k))\ind\{ L_k X(t_k) =v_k \}\ind_{A_j[X^{t_k}]} \right]$ as $t\uparrow t_k$. Moreover, $\ind_{A_j[X^t]\setminus A_j[X^{t_k}]} \downarrow 0$ as $t\uparrow t_k$ and thus the second term tends to $0$ by Slutsky's theorem. Hence, by \eqref{eq:DefPstar}
		\begin{equation*}
			\begin{aligned}
				\lim_{j\to\infty} \lim_{t\uparrow t_k} \bb{E}^{h} \left[ \frac{Z(t)}{h(t,X(t))} \ind_{A_j[X^t]}\right] &= \frac{1}{h(0,x_0)}\lim_{j\to\infty}\lim_{t\uparrow t_k} \bb{E}\left[Z(t)\ind_{A_j[X^t]}\right] \\
				&= \frac{1}{h(0,x_0)}\lim_{j\to\infty}\bb{E}\left[ g(t_k,X(t_k))\ind\{ L_k X(t_k) =v_k \} \ind_{A_j[X^{t_k}]} \right]\\
				&= \frac{ \bb{E}\left[ g(t_k,X(t_k))\ind\{ L_kX(t_k)=v_k \}\ind_{A[X^{t_k}]} \right] }{h(0,x_0)}.
			\end{aligned}
		\end{equation*}
	\end{proof}
	
	\begin{lem}
		\label{lem:ProofOfEquivalence-lem3}
		For all $j$,
		\[ \sup_{t<t_n} \exp\left( \int_0^t \frac{\scr{A}g}{g}(s,X(s))\dd s \right) \ind_{A_j[X^t]} \leq \exp(jt_n).  \]
	\end{lem}
	\begin{proof}
		Since $\scr{A}=\partial_t+\scr{L}$, 
		\[ \frac{\scr{A}g}{g}(s,x) = \partial_s\log g(s,x) + \frac{\scr{L}g}{g}(s,x). \]
		Note that 
		\[ \frac{\scr{L}g}{g}(s,x) = \sum_{\ell\in\scr{R}} \frac{\lambda_\ell(s,x)\left(g(s,x+\xi_\ell)-g(s,x)\right)}{g(s,x)} = V(s,x)-\sum_{\ell\in\scr{R}}\lambda_\ell(s,x) \leq V(s,x). \]
		By \Cref{lem:limitassumption-htilde}, $\partial_t\log g \leq 0$ and hence, for all $t<t_n$ and on $A_j[X^t]$, 
		\[ \int_0^t \frac{\scr{A}g}{g}(s,X(s))\dd s =  \int_0^t \partial_s\log g(s,X(s))\dd s + \int_0^t \frac{\scr{L}g}{g}(s,X(s))\leq jt \leq jt_n. \]
		
	\begin{proof}[Proof of \Cref{prop:sets}]
By \eqref{eq:behavior-loghtilde},
\begin{equation}
	\label{eq:limit-behavior-guided-rates}
		\lim_{t\uparrow t_k} \frac{g(t,x+\xi_\ell)}{g(t,x)} = \begin{cases} \frac{g(t_k,x+\xi_\ell)}{g(t_k,x)} & \text{if }L_kx=L_k(x+\xi_\ell)=v_k \\
		0 & \text{if }d_k(v_k,L_kx)>d_k(v_k,L_k(x+\xi_\ell) \\
		\infty & \text{if }d_k(v_k,L_kx) < d_k(v_k,L_k(x+\xi_\ell) \end{cases}.
	\end{equation}

	And therefore 
	\begin{equation}
		\label{eq:limit-behaviour-guided-rates}
		\lim_{t\uparrow t_k} \lambda_\ell^{g}(t,x) = \begin{cases}
			\lambda_\ell(t_k,x)\frac{g(t_k,x+\xi_\ell)}{g(t_k,x)} & \text{if }L_kx=L_k(x+\xi_\ell)=v_k \\
			0 & \text{if }d_k(v_k,L_k(x+\xi_\ell)) > d_k(v_k,L_kx )\\
			\infty & \text{if }d_k(v_k,L_k(x+\xi_\ell)) < d_k(v_k,L_kx )
		\end{cases}.
	\end{equation}
Now, for $t\in[t_{k-1},t_k)$, define $B_t^k = \{ L_kX(s)=v_k\text{ for all }s\in[t,t_k]\}$. Then it follows that for all $k$ and $t\in[t_{k-1},t_k)$,
\[ B_t^k \subseteq \left\{ \sup_{t_{k-1}\leq s<t_k}\, \sum_{\ell\in\scr{R}}\lambda_\ell^{g}(s,X(s)) <\infty \right\}. \]
Hence 
\[ 
	 A_n = \bigcup_{k=1}^n \bigcap_{t_{k-1}\leq t<t_k} B_t^k \\
	 \subseteq \bigcup_{k=1}^n \left\{  \sup_{t_{k-1}\leq s<t_k}\, \sum_{\ell\in\scr{R}}\lambda_\ell^{g}(s,X(s)) <\infty  \right\} = A[X^{t_n}] . \]
\end{proof}
	\end{proof}
	
			\begin{proof}[Proof of \Cref{thm:Equivalence}]
			Observe that the result follows immediately if we show that for $k=1,\dots,n$, $\bb{P}^g(L_k(X(t_k)=v_k) >0$. Now note The probability distribution of the next reaction time $\tau$ satisfies, for $t_{k-1}<s<t<t_k$ and $x\in \bb{S}$,
			\[ \bb{P}^{g}\left(\tau > t \mid X(s)=x\right) = \exp\left(-\int_s^t \sum_{\ell\in\scr{R}}\lambda_\ell^{g}(u, x)\dd u\right). \]
			
			Note that in \eqref{eq:limit-behaviour-guided-rates}, the convergence/divergence is at an exponential rate and therefore, by \eqref{eq:limit-behaviour-guided-rates} and \Cref{ass:greedy}, 
			\[ \bb{P}^{g}\left(\tau > t_k \mid X(s)=x\right) = 0 \]
			for all $s<t_k$ whenever $L_kx\neq v_k$. Hence, if $L_kX(s)\neq v_k$ for any $s<t_k$, $X$ jumps before time $t_k$ with probability $1$. Since this holds for any $s<t_k$, $X$ jumps an infinite amount of times if $L_kX$ does not hit $v_k$. 
			
			We now consider the jumps that take $L_kX$ closer to $v_k$ with respect to the metric $d_k$. Let 
			\[ \tau_+ = \inf\{t\geq s \colon d_k(v_k,L_kX(t))<d(v_k, L_kX(s))\} .\]
			Then the distribution of $\tau_+$ satisfies for any $x$ such that $L_kx\neq v_k$ and $t_{k-1}<s<t<t_k$,
			\[ \bb{P}^{g}(\tau_+ > t\mid X(s)=x) \leq \exp\left(-\int_s^t \sum_{\ell\in\scr{R}}\lambda_\ell^{g}(u,x)\ind_{\{d_k(v_k, L_k(x+\xi_\ell))<d_k(v_k,L_k x)\}} \dd u \right) .\]
			By \Cref{ass:greedy}, there are nonzero entries in this summation that tend to $\infty$ exponentially as $u\uparrow t_k$ and thus 
			\[ \bb{P}^{g} \left(\tau_+>t_k\mid X(s)=x\right)=0 \]
			for any $s<t_k$ and $L_kx\neq v_k$. In other words, for any $s<t_k$ and $L_k x\neq v_k$, 
			\[ \bb{P}^{g}\left( \exists t\in [s,t_k]\text{ such that }d_k(v_k,L_k X(t))< d(v_k,L_k x) \mid X(s)=x\right)=1. \]
			Now $\bb{S}$ is a discrete lattice and thus any set $\bb{S}\cap K$, where $K\subseteq\bb{R}^d$ is compact, contains only finitely many points. Hence, since $L_k X$ keeps jumping with probability $1$ to points strictly closer to $v_k$ if it does not hit $v_k$, it must hit $v_k$ with probability $1$ and thus
			\[ \bb{P}^{g}\left( \exists t\in [s,t_k]\text{ such that }L_kX(t)=v_k\mid X(s)=x\right)=1 \]
			for any $x$ with $L_kx\neq v_k$. Since this obviously also holds for $L_kx=v_k$, we have by the law of total probability 
			\begin{equation}
				\label{eq:proof-hitsx_T} 
				\bb{P}^{g}\left( \exists t\in [s,t_k]\text{ such that }L_kX(t)=v_k\right)=1 .
			\end{equation}
			Now denote the collection of events $B_s = \bigcup_{s\leq t\leq t_k} \{L_kX(t)=v_k\}$, $s\in[t_{k-1},t_k]$. By \eqref{eq:proof-hitsx_T}, $\bb{P}^{g}(B_s)=1$ for all $s$. Since $B_s \downarrow \{L_kX(t_k)=v_k\}$ as $s\uparrow t_k$, we have by monotonicity of  measures that $\bb{P}^{g}(L_kX(t_k)=v_k)=\lim_{s\uparrow t_k}\bb{P}^{g}(B_s)=1$. 	
		\end{proof}

	\section{Additional Lemmas}
	 \begin{lem}
		\label{lem:Bounds-on-htilde}
		Let $g$ be defined by \eqref{eq:htilde-noHF}. Then for all $x$, $g(t,x)$ is bounded and bounded away from zero uniformly in $t$. 
	\end{lem}
	\begin{proof}
		Note 
		\[ (v-Lx)^\T M(t)(v-Lx) \geq \lambda_{\min}(M(t))\abs{v-Lx}^2 = \left(  \lambda_{\max}(M^\dagger(t))\right)^{-1}\abs{v-Lx}^2.\]
		Since for all symmetric matrices $A,B$,
		\[\lambda_{\min}(A)+\lambda_{\min}(B) \leq\lambda_{\min}(A+B)\leq \lambda_{\max}(A+B)\leq\lambda_{\max}(A)+\lambda_{\max}(B),\]
		we have that 
		\[ \left(\lambda_{\max}(M^\dagger(t))\right)^{-1} \geq \left( \lambda_{\max}(C)+\lambda_{\max}(LaL^\T)(T-t)\right)^{-1} .\]
		Now $LaL^\T$ is positive definite and thus the right hand side is decreasing in $t$. Hence, a lower bound is given by $ \left( \lambda_{\max}(C)  + T\lambda_{\max}\left(LaL^\T\right)\right)^{-1}$. We conclude 
		\[ g(t,x) \leq \exp\left( -\frac12 \abs{v-Lx}^2\left(\lambda_{\max}(C)  + T\lambda_{\max}\left(LaL^\T\right) \right)^{-1} \right). \]
		In a similar way, we can show 
		\[ g(t,x) \geq \exp\left( -\frac12 \abs{v-Lx}^2\lambda_{\min}(C)^{-1}\right). \]
	\end{proof}

\begin{cor}[Multiple observations]
\label{lem:bounds-on-htilde-multiple}
			Let $g$ be defined by  \eqref{eq:htilde-multiple-alternative}. Then for all $x$, $g(t,x)$  is bounded and bounded away from zero uniformly in $t$. 
\end{cor}
\begin{proof}
	The proof of \Cref{lem:Bounds-on-htilde} can be repeated using, for $t\in(t_{k-1},t_k]$, 
	\[ \abs{v(t)-L(t)x}^2 = \sum_{j=k}^n\abs{v_j-L_jx}^2  \in \left[ \abs{v_n-L_nx}^2 , \sum_{j=1}^n \abs{v_j-L_jx}^2 \right].\]
\end{proof}

\section{Change of generator under change of measure}\label{sec:change of generator}
Here, we briefly recap the argument given in Section 4.1 of \cite{palmowski2002}. 
Fix $g$ and define $\{\tilde{\mathbb{P}}_t,\, t\ge 0\}$ by 
\[ \frac{\dd \tilde{\mathbb{P}}_t}{\dd \bb{P}_t} = E^g(t) := \frac{g(t,X(t))}{g(0,x_0)} \exp\left(-\int_0^t \frac{\mathcal{A} g}{g}(s,X(s)) \dd s \right),\] where $\mathcal{A}$ is the generator of the space-time process under $\bb{P}$. We aim to show that under $\tilde{\bb{P}}$ the generator of this process is given by $\tilde{\mathcal{A}}$ where $g\tilde{\mathcal{A}}f= \mathcal{A}(fg)-f \mathcal{A}g$. For this, it suffices to prove that 
\[ \tilde{D}^f(t) := f(t,X(t)) - \int_0^t (\tilde{\mathcal{A}}f)(s, X(s)) \dd s \]
is a local martingale under $\tilde{\bb{P}}$. To show this, by Lemma 3.1 in \cite{palmowski2002} $\tilde{D}^f(t)$ is a local $\tilde{\bb{P}}$-martingale if and only if $\tilde{E}^f(t)$ is a local $\tilde{\bb{P}}$-martingale, where $\tilde{E}^f(t)$ is defined as $E^f(t)$ but with $\mathcal{A}$ replaced by $\tilde{\mathcal{A}}$. To show that $\tilde{E}^f(t)$ is a local $\tilde{\bb{P}}$-martingale, it suffices to show that $\tilde{E}^f(t) E^g(t)$ is a local $\bb{P}$-martingale (Cf.\ Lemma 4.1 in \cite{palmowski2002}). 
Straightforward calculus gives
\[ \tilde{E}^f(t) E^g(t) = E^{fg}(t), \]
which is indeed a local $\bb{P}$ martingale.

	\bibliographystyle{authordate1}
	\bibliography{biblio}
\end{document}